\pdfoutput=1
\RequirePackage{ifpdf}
\ifpdf 
\documentclass[pdftex]{sigma}
\else
\documentclass{sigma}
\fi

\numberwithin{equation}{section}

\newtheorem{thm}{Theorem}[section]
\newtheorem*{Theorem*}{Theorem}
\newtheorem{cor}[thm]{Corollary}
\newtheorem{lem}[thm]{Lemma}
\newtheorem{prop}[thm]{Proposition}

\newtheorem{theoremalpha}{Theorem}

\theoremstyle{definition}

\newtheorem{Example}[thm]{Example}
\newtheorem{rem}[thm]{Remark}

\usepackage[all]{xy} 
\usepackage{tikz}
\usetikzlibrary{cd}
\usetikzlibrary{decorations.pathmorphing}

\usepackage{enumitem}

\usepackage[reftex]{theoremref}

\usepackage[normalem]{ulem}

\DeclareMathOperator{\curv}{curv}

\DeclareMathOperator{\diam}{diam}

\DeclareMathOperator{\Iso}{Isom}

\DeclareMathOperator{\img}{Image}
\DeclareMathOperator{\codim}{codim}

\DeclareMathOperator{\Home}{\mathrm{Homeo}}

\DeclareMathOperator{\vol}{vol}
\DeclareMathOperator{\rad}{rad}
\DeclareMathOperator{\Pol}{Pol}

\usepackage{tikz}

\newcommand{\N}{\mathbb{N}} 
\newcommand{\R}{\mathbb{R}} 
\newcommand{\Z}{\mathbb{Z}} 
\newcommand{\C}{\mathbb{C}} 

\newcommand{\Sp}{S} 
\newcommand{\RP}{\mathbb{R}P} 
\newcommand{\Id}{\mathrm{Id}}

\newcommand{\Tau}{\mathrm{T}}

\newcommand{\prin}{\mathrm{prin}}

\newcommand{\fol}{\mathcal{F}}

\newcommand{\p}{\pi}
\newcommand{\Rho}{\mathrm{P}}

\newcommand\restr[2]{{
		\left.\kern-\nulldelimiterspace 
		#1 
		\right|_{#2} 
}}

\begin{document}
\allowdisplaybreaks

\newcommand{\arXivNumber}{2407.03534}

\renewcommand{\PaperNumber}{106}

\FirstPageHeading
	
\ShortArticleName{Myers--Steenrod Theorems for Metric and Singular Riemannian Foliations}
	
\ArticleName{Myers--Steenrod Theorems for Metric\\ and Singular Riemannian Foliations}
	
\Author{Diego CORRO~$^{\rm ab}$ and Fernando GALAZ-GARC\'IA~$^{\rm c}$}
	
\AuthorNameForHeading{D.~Corro and F.~Galaz-Garc\'ia}
	
\Address{$^{\rm a)}$~Fakult\"at f\"ur Mathematik, Karlsruher Institut f\"ur Technologie, Germany}

\Address{$^{\rm b)}$~School of Mathematics, Cardiff University, UK}
\EmailD{\mail{diego.corro.math@gmail.com}}
\URLaddressD{\url{https://www.diegocorro.com/}}
	
\Address{$^{\rm c)}$~Department of Mathematical Sciences, Durham University, UK}
\EmailD{\mail{fernando.galaz-garcia@durham.ac.uk}}
\URLaddressD{\url{https://www.durham.ac.uk/staff/fernando-galaz-garcia/}}
	
\ArticleDates{Received November 04, 2024, in final form December 01, 2025; Published online December 16, 2025}	

\Abstract{We prove that the group of isometries preserving a metric foliation on a closed Alexandrov space $X$ is a closed subgroup of the isometry group of $X$. We obtain a sharp upper bound for the dimension of this subgroup and show that, when equality holds, the foliations that realize this upper bound are induced by fiber bundles whose fibers are round spheres or projective spaces. As a corollary, singular Riemannian foliations that realize the upper bound are induced by smooth fiber bundles whose fibers are round spheres or projective spaces.}

\Classification{53C12; 53C20; 53C21; 53C23; 53C24; 51K10}
	
\section{Main results}
The Myers--Steenrod theorem \cite{MyersSteenrod} states that the isometry group of a Riemannian $n$-manifold is a Lie group whose dimension is at most $n(n+1)/2$. When the manifold is compact, its isometry group must also be compact, as established by van Dantzig and van der Waerden~\cite{van_Dantzig_1928}. The application of the theory of compact transformation groups in Riemannian geometry \mbox{\cite{AlexandrinoBettiol2015,Bredon,Grove2002,Kobayashi}} is grounded on these two fundamental results, which also hold for other classes of metric spaces, such as Finsler manifolds \cite{DengHour2002}, Alexandrov spaces \cite{FukayaYamaguchi1994, GalazGarciaGuijarro2013} or $\mathrm{RCD}$ spaces \cite{GuijarroSantos2018, Sosa2018}.

Singular Riemannian foliations generalize both isometric compact Lie group actions and Riemannian submersions, which induce decompositions into embedded submanifolds of lower dimension, and represent a generalized notion of symmetry on Riemannian manifolds \cite{AlexandrinoRadeschi2017, Corro, CorroMoreno2022, GalazGarciaRadeschi2015, Moreno2019}. Not all singular Riemannian foliations stem from Lie group actions
(see, for example,~\cite{Radeschi2014}). Nevertheless, certain isometries of a Riemannian manifold $M$ with a singular Riemannian foliation~$\fol$ may induce residual symmetry by preserving the foliation's leaves. These \emph{foliated isometries} generate isometries of the leaf space $M/\fol$. Our first main result is an analog of the Myers--Steenrod theorem for the group of foliated isometries of a Riemannian manifold with a~singular Riemannian foliation.

\begin{theoremalpha}\label{T:MAIN_THM_SRF}
		Let $(M,\fol)$ be a singular Riemannian foliation with closed leaves on a complete connected Riemannian manifold. Then the following assertions hold:
		\begin{enumerate}[label=$(\roman*)$]\itemsep=0pt
			\item\label{C:MAIN_THM_i} The group $\Iso(M,\fol)$ of foliated isometries is a Lie group and is compact if $M$ is compact.
			\item\label{C:MAIN_THM_ii} If $M$ has dimension $n\geq 1$ and $\fol$ has codimension $0\leq k\leq n$, then
				\begin{equation}
				\dim(\Iso(M,\fol))\leq\frac{k(k+1)}{2} + \frac{(n-k)(n-k+1)}{2}.\tag{\emph{a}}\label{EQ:DIM_BOUND}		
				\end{equation}
			\item\label{C:MAIN_THM_iii} If equality holds in inequality~\eqref{EQ:DIM_BOUND}, then $M$ is foliated-diffeomorphic to a fiber bundle $F\to M\to B$, where $B$ is diffeomorphic to $\R^k$, $\RP^k$, or $S^k$, and the fibers are round spheres or real projective spaces, Euclidean space, or a hyperbolic space. Moreover, when equipped with the leaf-projection metric induced by $M$, the base space $B$ is isometric to $k$-dimensional Euclidean space, hyperbolic space with constant negative sectional curvature, a round real projective space, or a round sphere. 	
		\end{enumerate}
	\end{theoremalpha}

The leaf space $M/\fol$ of a singular Riemannian foliation $(M,\fol)$ with closed leaves has a~natural distance function that makes $M/\fol$ a locally compact length space. The curvature of~$M/\fol$ is locally bounded from below in the triangle comparison sense, which links the geometry of the leaf space with that of the manifold. This fact plays a central role in proving Theorem~\ref{T:MAIN_THM_SRF}.
For an isometric action of a compact Lie group $G$ on a complete Riemannian manifold $M$ with sectional curvature uniformly bounded from below by $k\in \R$, the orbit space $M/G$ equipped with the orbital distance function is an Alexandrov space with curvature bounded below by $k$. Imposing further conditions, such as positive or non-negative sectional curvature on $M$, leads to significant constraints on the manifold's topology and has been an active research topic in Riemannian geometry \cite{Grove2002,Grove2017,Wilking2007}.
Note that the orbit projection map $\pi\colon M\to M/G$ is a~proper \emph{submetry}, i.e., for every $p\in M$, any closed metric ball $B(p,r)$ of radius $r$ centered at $p$ maps onto the metric ball $B(\pi(p),r)$ in $M/G$. Submetries, introduced by Berestovskii in \cite{Berestovskii1987} as a metric generalization of Riemannian submersions, have been the focus of systematic study in metric geometry \cite{BeretovskiiGuijarro2000,GuijarroWalschap2011,KapovitchLytchak2022,Lytchak,Lytchak2024,LytchakWilking2024}.

Motivated by the preceding considerations, we also investigate the foliated isometries of the metric foliations $\fol$ whose leaves are the fibers of a submetry $\pi\colon X\to Y$ between Alexandrov spaces. Each leaf $\pi^{-1}(y)$ is closed, since it is the inverse image of the closed set $\{y\}$ under the continuous map $\pi$.
We consider the group $\Iso(X,\fol)$ of foliated isometries and obtain the following analog of Theorem~\ref{T:MAIN_THM_SRF}. For parts~(ii)--(iii), we assume $\pi$ is proper, in which case every fiber is compact.

	\begin{theoremalpha}\label{T:MAIN_THM}
		Let $\pi\colon X\to Y$ be a submetry between connected Alexandrov spaces and set $\fol = \big\{\p^{-1}(y)\mid y\in Y\big\}$. Then the following assertions hold:
		\begin{enumerate}[label=$(\roman*)$]\itemsep=0pt
			\item\label{T:MAIN_THM_i} The group $\Iso(X,\fol)$ of foliated isometries of $(X,\fol)$ is a Lie group and is compact if $X$ is compact.
 \end{enumerate}
 Assume further that $\pi$ is proper for parts $(ii)$ and $(iii)$.
 \begin{enumerate}[label=$(\roman*)$]\itemsep=0pt
 \setcounter{enumi}{1}
			\item\label{T:MAIN_THM_ii}
 If $X$ has dimension $n\geq 1$ and $Y$ has dimension $0\leq k \leq n$, then
			\begin{equation}
				\dim(\Iso(X,\fol))\leq\frac{k(k+1)}{2} + \frac{(n-k)(n-k+1)}{2}.\tag{\emph{b}}\label{EQ:DIM_BOUND_ALEX}
			\end{equation}
\item\label{T:MAIN_THM_iii}If equality holds in inequality~\eqref{EQ:DIM_BOUND_ALEX}, then
$\pi$ is the composition of a submetry $\tilde{\pi}\colon X\to Z$ with connected fibers, and a submetry $\pi_D\colon Z\to Y$ whose fibers are finite discrete spaces. The spaces $X$, $Y$, $Z$ are Riemannian manifolds, $X$ is homogeneous, and the submetry $\tilde{\pi}\colon {X\to Z}$ is a smooth Riemannian submersion. Moreover, the conclusions of Theorem~\ref{T:MAIN_THM_SRF}\,\ref{C:MAIN_THM_iii} hold for the foliation induced by $\tilde{\pi}$. In general the fibers of $\pi$ are a finite disjoint union of round spheres or projective spaces, the base space $Y$ is isometric to the $k$-dimensional Euclidean space, hyperbolic space with constant negative sectional curvature, a round real projective space, or a round sphere.
More specifically, when $Y$ is homeomorphic to $\R^k$ for $k\geq0$, or $Y$ is homeomorphic to $\Sp^k$ for $k\geq 2$, then the fibers of $\pi$ are connected. When $Y$ is homeomorphic to $\RP^k$ for $k\geq 2$, the fibers of $\pi$ have at most two connected components. When $Y$ is homeomorphic to $\Sp^1$, then the fibers of $\pi$ can have $m\geq 1$ connected components. 	
\end{enumerate}
	\end{theoremalpha}

The upper bound for the dimension of the group of foliated isometries in Theorems~\ref{T:MAIN_THM_SRF} and \ref{T:MAIN_THM} consists of two summands. The first summand, $k(k+1)/2$, bounds the dimension of the image of the Lie group morphism $\Psi\colon \Iso(M,\fol)\to \Iso(M/\fol)$ induced by the projection $\pi$ (see Section~\ref{s:foliated.maps}).
This image consists of the isometries of $M$ that descend to non-trivial isometries of the leaf space. The second summand, $(n-k)(n-k+1)/2$, bounds the dimension of the group of foliated isometries $h$ that leave the leaves invariant, i.e., $h(L)\subset L$ for any leaf $L\in \fol$.

Applying Theorem~\ref{T:MAIN_THM_SRF} to the trivial foliations $\fol = \{M\}$ or $\fol = \{{p}\mid p\in M\}$ consisting, respectively, of a single leaf or leaves that contain only one point, yields the classical upper bound~${n(n+1)/2}$ in the Riemannian Myers--Steenrod theorem. For Alexandrov spaces, Theorem~\ref{T:MAIN_THM} applied to the trivial submetries $\mathrm{id}\colon X \to X$ and $\pi\colon X \to \{\mathrm{pt}\}$ yields the bound $n(n+1)/2$ in the Myers--Steenrod theorem for Alexandrov spaces in \cite{GalazGarciaGuijarro2013}.

As in the Myers--Steenrod theorem for the setting of Riemannian, Alexandrov, and $\mathrm{RCD}$-spaces, the upper bound in Theorems~\ref{T:MAIN_THM_SRF} and \ref{T:MAIN_THM} is sharp. Let $(M,g)$ be the Riemannian product~${(N\times P, g_1\times g_2)}$ of two Riemannian manifolds $\bigl(N^{n-k},g_1\bigr)$, $\bigl(P^k,g_2\bigr)$ isometric to round spheres or round projective spaces
with $k,n-k\geq 1$,
and the Riemannian foliation $\fol$ whose leaves are $N$. Then $\dim(\Iso(M,g)) =\dim( \Iso(N,g_1))+\dim(\Iso(P,g_2))$ (see \cite[Corollary 1]{EschenburgHeintze1998}), and thus $\dim(\Iso(M,\fol)) = \dim(\Iso(M))$. This realizes the upper bound in Theorem~\ref{T:MAIN_THM_SRF}\,\eqref{EQ:DIM_BOUND} and Theorem~\ref{T:MAIN_THM}\,\eqref{EQ:DIM_BOUND_ALEX}.

As a nontrivial bundle example, consider the Hopf fibration $\fol$ given by the orbits of the scalar multiplication action of $S^1\subset \C$ on the unit sphere $S^3\subset \C^2$.
Since this $S^1$-action commutes with the standard $\mathrm{U}(2)$-action on $S^3$, we have $\mathrm{U}(2)\subset \Iso\bigl(S^3,\fol\bigr)$.
By Theorem~\ref{T:MAIN_THM_SRF}\,\ref{C:MAIN_THM_ii},
\[
\dim \mathrm{U}(2) = 4\leq \dim \bigl(\Iso\bigl(S^3,\fol\bigr)\bigr)\leq 4.
\]
Hence, $\Iso\bigl(S^3,\fol\bigr)$ has maximal dimension while the foliation is given by a nontrivial fiber bundle.

When equality holds in inequality~\eqref{EQ:DIM_BOUND}, $M$ is a fiber bundle $F\to M\to B$ and we may assume that the leaves $F$ carry a metric of constant sectional curvature equal to one. The main theorems in \cite{CorroGuentherGarciaKordass2020,FarrellGangKnopfOntaneda2017} imply that, when $B$ is not diffeomorphic to $\R^k$, the bundle has a linear structure group. In particular, when the leaves are diffeomorphic to $S^{n-k}$, the classification of such bundles is the same as the classification of vector bundles over $B$. In the case when the leaves are diffeomorphic to $\RP^{n-k}$, in some cases depending on the values of $k$, we can show that the conditions of \cite[Theorem 1]{BeckerGottlieb1973} hold, and thus the classification reduces to the classification of vector bundles over $B$. In general when the leaves are diffeomorphic to $\RP^{n-k}$, by \cite{CorroGuentherGarciaKordass2020}, the structure group is the so-called \emph{projective orthogonal group} $\mathrm{PO}(n-k)=\mathrm{O}(n-k)/\{\mathrm{Id},-\mathrm{Id}\}$, and thus the classification corresponds to $[B,B\mathrm{PO}(n-k)]$, the collection of maps from $B$ to the classifying space $B\mathrm{PO}(n-k)$ up to homotopy.

When $B$ is diffeomorphic to $\R^k$, the manifold $M$ is diffeomorphic to $\R P^{n-k}\times \R^k$ or $S^{n-k}\times \R^k$, since $\R^k$ is contractible. Observe that in Theorem~\ref{T:MAIN_THM}\,\ref{T:MAIN_THM_iii} in the case when $Y$ is homeomorphic to $\R^k$, we conclude that $X$ is homeomorphic to $S^{n-k}\times\R^k$ or $\RP^{n-k}\times\R^k$.

The map $\Psi\colon\Iso(M,\fol)\to \Iso(M/\fol)$ is generally not surjective. However, if $M$ is an $n$-dimensional Euclidean vector space with a foliation $\fol$ induced by a linear isometric action by a~compact Lie group $G$, any isometry in $\Iso(M/G)^0$ lifts to a foliated isometry in~$\Iso(M,G)$~\cite{Mendes2020}. Here, $\Iso(M/G)^0$ denotes the connected component containing the identity. The proof of Theorem~\ref{T:MAIN_THM} yields an upper bound on the dimension of the orbits of $G$ in terms of the dimension of $M/G$. Specifically, if $M/G$ has dimension $0<k<n$, then an orbit has dimension at most~${(n-k)(n-k+1)/2}$.

\begin{rem}
We point out that in the proof of Theorem~\ref{T:MAIN_THM_SRF}, we only use the transnormal properties of the closed foliation. Nonetheless, we state our results for singular Riemannian foliations, as there are no known examples of transnormal systems that are not smooth foliations, and it is conjectured that any transnormal system is a singular Riemannian foliation (see \cite[Final remarks]{Wilking2007}).
See also the work of Lytchak and Wilking \cite{LytchakWilking2024} for the smoothness of submetries between Riemannian manifolds.
\end{rem}

\begin{rem}
In the proof of Theorems~\ref{T:MAIN_THM_SRF}\,\ref{C:MAIN_THM_i} and \ref{T:MAIN_THM}\,\ref{T:MAIN_THM_i}, we crucially rely on the assumption that the leaf and fiber spaces are Hausdorff (see, for example, the proof of Proposition~\ref{P: Iso(Y) closed implies Iso(X,F) closed}).
To prove Theorem~\ref{T:MAIN_THM_SRF}\,\ref{C:MAIN_THM_ii}, we require the leaf space to be an Alexandrov space. Both conditions are satisfied if the singular Riemannian foliation has closed leaves. However, this requirement can be relaxed by asking that the leaves of the foliation be globally equidistant instead of just locally equidistant.
\end{rem}

\begin{rem}
We do not know if the upper bound on the dimension stated in Theorem~\ref{T:MAIN_THM_SRF}\,\ref{C:MAIN_THM_ii} holds for singular Riemannian foliations without closed leaves. Even if such a bound exists, we cannot obtain a rigidity conclusion as in Theorem~\ref{T:MAIN_THM_SRF}\,\ref{C:MAIN_THM_iii}. This is illustrated by the irrational flow on the $2$-dimensional flat torus: the group of foliated isometries has dimension two, which is the dimension of the isometry group of the $2$-dimensional flat torus. In contrast, when we assume that the leaves are closed and the singular Riemannian foliation is not trivial, Theorem~\ref{T:MAIN_THM_SRF}\,\ref{C:MAIN_THM_iii} implies that the foliation is given by a circle bundle over the circle.
\end{rem}

\begin{rem}
In Theorem~\ref{T:MAIN_THM}\,\ref{T:MAIN_THM_iii}, there are examples of foliations with disconnected fibers. Namely, consider the product of a round $2$-sphere with a round $3$-sphere, and consider the $3$-sphere as leaves. As remarked above, the group of foliated isometries has maximal dimension. Now consider the antipodal action of $\Z_2$ on $S^2$. With this, we get a fibration of $S^2\times S^3$ over~$\R P^2$ whose fibers are two disjoint copies of a round $S^3$. The difference with Theorem~\ref{T:MAIN_THM_SRF} is that we ask the leaves of a singular Riemannian foliation to be connected.
\end{rem}

\begin{rem}
\label{R:classification.problem}
 We note a natural classification problem: \emph{classify, up to foliated isometry, all pairs~$(X,\fol)$ with $\Iso(X,\fol)$ of maximal dimension.} Our results provide initial structure constraints for this question.
\end{rem}

Our article is organized as follows. Section~\ref{S:PRELIMINARIES} presents basic material on singular Riemannian foliations, Alexandrov spaces, and submetries, as well as auxiliary results used in the proofs of our main theorems. Section~\ref{S: proof of Main corollary} contains the proof of Theorem~\ref{T:MAIN_THM_SRF}. Finally, in Section~\ref{S:PROOF_MAIN_THM}, we prove Theorem~\ref{T:MAIN_THM}.

\section{Preliminaries}
\label{S:PRELIMINARIES}

In this section, we collect basic definitions and results on singular Riemannian foliations, Alexandrov spaces, and submetries we will use in the proof of our main theorems. We refer the reader to \cite{AlexandrinoBriquetToeben2013,BuragoBuragoIvanov,GromollWalschap2009,KapovitchLytchak2022,Lytchak} for further basic results on these subjects. We will assume all spaces are connected, unless stated otherwise.

\subsection{Singular Riemannian foliations}
\label{S: Singular Riemannian foliations}
A \emph{singular Riemannian foliation} $(M,\fol)$ on a complete Riemannian manifold $M$ is a partition of the manifold into a collection $\fol = \{L_x\mid x\in M\}$ of disjoint connected, complete, immersed submanifolds $L_x$, called \emph{leaves}, satisfying the following conditions:
\begin{enumerate}\itemsep=0pt
\item[(i)] If $\gamma\colon [a,b] \to M$ is a geodesic perpendicular to the leaf $L_{\gamma(a)}$, then $\gamma$ is perpendicular to~$L_{\gamma(t)}$ for all $t\in [a,b]$.
\item[(ii)] For each $p\in M$, there exists local smooth vector fields spanning the tangent spaces of the leaves.	
\end{enumerate}

We call any leaf of maximal dimension a \emph{regular} leaf; leaves that are not regular are called \emph{singular}. Given any Riemannian manifold $M$, the foliation consisting of one leaf $\fol = \{M\}$ and the foliation where each leaf consists of just a point $\fol = \{\{x\}\mid x\in M\}$ are trivial examples of singular Riemannian foliations.
Other examples of singular Riemannian foliations are given by the partition of a complete Riemannian manifold $M$ by the orbits of an isometric action of a~compact Lie group. The partition of $M$ into the fibers of a Riemannian submersion~${f\colon M\to N}$ yields a further example of a singular Riemannian foliation. Note that there are infinitely many examples of singular Riemannian foliations which are not given by group actions nor from Riemannian submersions \cite{FerusKarcherMuenzner1981,Radeschi2014}.
	
Let $(M,\fol)$ be a singular Riemannian foliation. Then we have a singular distribution $H\subset M$, called the \emph{horizontal distribution}, given by setting $H_x = \nu_x L_x$, the normal tangent space to the leaf $L_x$ at $x$. The \emph{codimension of the foliation}, denoted by $\codim (\fol)$, is the codimension of any regular leaf in $M$. We say $(M,\fol)$ is \emph{closed} if all leaves are closed in $M$.
The \emph{leaf space} of the foliation is the set of equivalence classes $M^\ast = M/\fol$ equipped with the quotient topology. We have a natural projection map $\pi\colon M\to M^\ast$ which is continuous with respect to the quotient topology. Given a subset $A\subset M$, we let $A^\ast = \pi(A)$.

Let $(M,\fol)$ be a closed singular Riemannian foliation. Fix $x\in M$ and consider the normal space $\nu_x L_x$ to the tangent space $T_xL_x\subset T_x M$ at $x$.
Next, for $\varepsilon > 0$ sufficiently small, define~${\nu_x^\varepsilon L_x = (\nu_x L_x) \cap B_{\varepsilon}(0)}$, where $B_{\varepsilon}(0)$ is the closed ball of radius $\varepsilon$ in $T_x M$.
Set $S_x = \exp_{x}(\nu_x^\varepsilon L_x)$.
The intersection of the leaves of $\fol$ with $S_x$ induces a foliation $\fol |_{S_x}$ on $S_x$ whose leaves are the connected components of the intersection between the leaves of $\fol$ and $S_x$.
Although $\fol|_{S_x}$ may not be a singular Riemannian foliation with respect to the induced metric of~$M$ on $S_x$ (the leaves of $\fol|_{S_x}$ may not be equidistant with respect to the induced metric), the pull-back foliation~${\fol^x = \exp_x^{\ast}(\fol|_{S_x})}$ is a singular Riemannian foliation on $\nu_x^\varepsilon L_x$ equipped with the Euclidean metric $g^\perp_x:=(g_x)|_{\nu_x L_x}$ (see \cite[Proposition~6.5]{Molino}).
The foliation $(\nu_x^\varepsilon L_x, \fol^x)$ is called the \emph{infinitesimal foliation at $x$}.

The infinitesimal foliation $(\nu_x^\varepsilon L_x, \fol^x)$ is invariant under homotheties fixing the origin (see~\cite[Lemma~6.2]{Molino}). Furthermore, the origin $\{0\}\subset \nu_x^\varepsilon L_x$ is a leaf of the infinitesimal foliation.
Since the leaves of $\fol^x$ are equidistant, the origin being a leaf implies that any leaf of $\fol^x$ is at a~constant distance from $\{0\}$. Therefore, each leaf of the infinitesimal foliation is contained in a~round sphere centered at the origin.
Hence, we may consider the infinitesimal foliation restricted to the unit normal sphere of $\nu_x L_x$, denoted by $S^{\bot}_x$, resulting in a foliated round sphere $\bigl(S^{\bot}_x, \fol^x\bigr)$ with respect to the standard round metric of $S^{\bot}_x$. This foliation is also called the \emph{infinitesimal foliation}. Henceforth, when referring to the ``infinitesimal foliation", we will mean $\bigl(S^{\bot}_x, \fol^x\bigr)$. Note that there is no loss of generality in doing so, since $(\nu_x L_x, \fol^x)$ is invariant under homothetic transformations and thus one may recover it from $\bigl(S^{\bot}_x, \fol^x\bigr)$.

Let $L$ be a closed leaf of a singular Riemannian foliation $(M, \mathcal{F})$,
and $\gamma\colon [0,1] \to L$ a piecewise smooth curve with $\gamma(0) = x $. By \cite{Mendes2020}, there exists a continuous map
$G\colon [0,1] \times \nu_p L \to \nu L$
such that
\begin{enumerate}\itemsep=0pt
 \item[(a)] \( G(t, v) \in \nu_{\gamma(t)} L \) for every \( (t, v) \in [0,1] \times \nu_x L \).
 \item[(b)] For every \( t \in [0,1] \), the restriction
\smash{$
 G|_{\{t\} \times \nu_x L}\colon\ \nu_x L \to \nu_{\gamma(t)} L
$}
 is a linear isometry preserving the leaves of \( \nu L \).
 \item[(c)] For every \( s \in \mathbb{R} \), \( \exp_{\gamma(t)} (sG(t, v)) \) belongs to the same leaf as \( \exp_x (sv) \).
\end{enumerate}

We denote by $\mathrm{O}\bigl(S^\perp_x,\fol^x\bigr)$ the group of \emph{foliated isometries} of the infinitesimal foliation, i.e., the isometries which preserve the infinitesimal foliation.
For each loop $\gamma$ at $x$, the map $G_\gamma\colon {\nu_x L\to \nu_x L}$ given by $G_\gamma(v) = G(1,v)$ is a foliated linear isometry (see \cite[Corollary~15]{MendesRadeschi2019}).
Therefore, we have a group homomorphism $\rho\colon\Omega(L,x)\to \mathrm{O}\bigl(S^\perp_x,\fol^x\bigr)$ from the loop space of $L_x$ at $x$ to the foliated isometries of the infinitesimal foliation by setting $\rho(\gamma) = G_\gamma$.

An isometry in $\mathrm{O}\bigl(S^\perp_x,\fol^x\bigr)$ may map a leaf to a different leaf. By $\mathrm{O}(\fol^{x})$ we denote the foliated isometries preserving the foliation, i.e., isometries $f\in \mathrm{O}\bigl(S^\perp_x,\fol^x\bigr)$ such that for any leaf~$\mathcal{L}$ of~$\bigl(S^\perp_x,\fol^x\bigr)$, we have $f(\mathcal{L})\subset \mathcal{L}$. The natural action of~$\mathrm{O}\bigl(S^\perp_{x},\fol^x\bigr)$ on the quotient~$\Sp^\perp_x/\fol^x$ has kernel~$\mathrm{O}(\fol^x)$.
Using the fact that if two loops $\gamma_1$ and $\gamma_2$ based at $x$ are homotopic, then~$G_{\gamma_1}^{-1}\circ G_{\gamma_2}$ is in the kernel of the action of $\mathrm{O}\bigl(\Sp^\perp_x,\fol^x\bigr)$ on $\Sp^\perp_x/\fol^x$, one may show that there is a well-defined group homomorphism
\[
	\rho\colon\ \pi_1(L,x)\to \mathrm{O}\bigl(\Sp^\perp_x,\fol^x\bigr)/\mathrm{O}(\fol^x),
\]
given by $\rho[\gamma] = [G_\gamma]$ (see, for example, \cite[Lemma 2.4 and Proposition 2.5]{Corro}). We define the \emph{holonomy of the leaf} $L$ as the image $\Gamma_L< \mathrm{O}\bigl(\Sp^\perp_x,\fol^x\bigr)/\mathrm{O}(\fol^{x})$ of $\pi_1(L,x)$ under the homomorphism $\rho$. When we consider the holonomy of a leaf $L_x$ through a point $x\in M$, we denote it by~$\Gamma_x$. A regular leaf $L$ is a \emph{principal leaf} if its holonomy group is trivial, and \emph{exceptional} otherwise. The set $M_\prin\subset M$ consisting of the union of principal leaves is an open and dense subset of $M$ (see for example \cite[Proposition 2.8]{Corro}).

\subsection{Alexandrov spaces}
\label{S: Alexandrov spaces}
An \emph{Alexandrov space $(X, d)$ with curvature bounded below by $k\in \R$} is a complete length space of finite Hausdorff dimension with curvature bounded below in the triangle comparison sense. Namely, for each $x\in X$, there is an open neighborhood $U\subset X$ of $x$, such that, for each geodesic triangle $\triangle$ contained in $U$, there exists a geodesic triangle $\widetilde{\triangle}$ in $M^2_k$, the $2$-dimensional model space of constant sectional curvature $k$, with edges having the same lengths as the edges of $\triangle$ satisfying the following condition: If $y$ is a vertex of $\triangle$, $\widetilde{y}$ is the corresponding vertex in~$\widetilde{\triangle}$, and~$w$ is any point in the opposing edge in~$\triangle$ with corresponding point~$\widetilde{w}$ in the opposing edge in~$\widetilde{\triangle}$, then
$
	d(y,w) \geq d_{M^2_k}(\widetilde{y},\widetilde{w})$.
As is customary, we will abbreviate \emph{curvature bounded below by $k$} by writing $\curv\geq k$. A~complete Riemannian manifold $M$ with a uniform lower bound for the sectional curvature is an example of an Alexandrov space.

Let $(X,d)$ be a metric space and fix three points $x,y,z\in X$. We define the \emph{comparison angle~$\measuredangle(yxz)$ at $x$} as
\[
	\measuredangle (y,x,z) = \arccos\left(\frac{d(x,y)^2+d(x,z)^2-d(y,z)^2}{2d(x,y)d(x,z)} \right).
\]
Now consider two continuous curves $c_1\colon [0,1]\to X$ and $c_2\colon [0,1]\to X$ with $c_1(0) = c_2(0) = {x\in X}$. We define the \emph{angle between $c_1$ and $c_2$} as
\[
	\angle (c_1,c_2) = \lim_{s,t\to 0}	\measuredangle (c_1(s),x,c_2(t)),
\]
provided the limit exists. When $X$ is an Alexandrov space and the curves $c_1$ and $c_2$ are geodesics, the angle $\angle (c_1,c_2)$ exists (see \cite[Proposition 4.3.2]{BuragoBuragoIvanov}).

Given two geodesics $c_1\colon [0,1]\to X$ and $c_2\colon [0,1]\to X$ in an Alexandrov space $X$ with common start point $x\in X$, we say that $c_1$ is \emph{equivalent} to $c_2$ if $\angle (c_1,c_2) =0$. Let $\widetilde{\Sigma}_x(X)$ the set of equivalence classes of geodesics starting at $x$. We define a metric on this set by setting the distance between two classes to be the angle formed between any two representatives of each class. The \emph{space of directions $\Sigma_x(X)$ of $X$ at $x$} is the metric completion of $\widetilde{\Sigma}_x(X)$. By \cite[Corollaries 7.10 and 7.11]{BuragoGromovPerelman1992}, the metric space $\Sigma_x(X)$ is an Alexandrov space of $\curv \geq 1$.

By \cite[Theorem~10.4.1]{BuragoBuragoIvanov}, any Alexandrov space $\Sigma$ with $\curv \geq 1$ has $\diam(\Sigma)\leq \pi$. Let~$(X,d)$ be a metric space with $\diam(X)\leq \pi$. The \emph{Euclidean cone over $X$}, denoted by $CX$, is the set~${X\times [0,\infty)/(x,0)\sim (y,0)}$ equipped with the metric
\[
	d_C([x,t],[y,s]) = \sqrt{t^2+s^2-2ts\cos(d(x,y))}.
\]
Given an Alexandrov space $X$ and a point $x\in X$, we define the \emph{tangent cone of $X$ at $x$} as $T_x X = C\Sigma_x(X)$. We denote the vertex of $T_xX$ by $0$.

The following basic example illustrates a connection between the theories of Alexandrov spaces and of closed singular Riemannian foliations.

\begin{Example}	
\label{ex:quotient.of.srf.is.alex.space}
The leaf space $M^\ast$ of a singular Riemannian foliation $(M,\fol)$ on a complete Riemannian manifold with closed leaves inherits a complete metric $d^\ast$ from $M$, known as the \emph{leaf-projection metric}.
For $x^\ast,y^\ast\in M^\ast$, the distance $d^\ast(x^\ast,y^\ast)$ is defined as
$
	 d^\ast(x^\ast,y^\ast) = d(L_x,L_y)$,
the distance between the leaves $L_x$ and $L_y$ considered as subsets of $M$.
Equipped with the metric $d^\ast$, the leaf space $M^\ast$ has curvature locally bounded below in the triangle comparison sense discussed above. Specifically, if $U\subset M$ is an open neighborhood with sectional curvature bounded below by $k_U\in \R$, then the projection $U^\ast \subset M^\ast$ has curvature bounded below by~$k_U$ (see \cite{LytchakThorbergsson2010}). The Hausdorff dimension of $M^*$ is equal to the codimension of $\fol$. Hence, if $M$ has sectional curvature uniformly bounded below by~$k\in \R$, and $\fol$ is closed, then $M^*$ is an Alexandrov space with $\curv \geq k$.
\end{Example}

\subsection{Submetries from proper metric spaces}
\label{S: Submetries between proper metric spaces}	
Below, we collect several results from \cite{Lytchak} on general submetries between Alexandrov spaces (see also \cite{GuijarroWalschap2011} and \cite{KapovitchLytchak2022}).

A map $p\colon X \to Y$ between two metric spaces is a \emph{submetry} if, for any $\varepsilon>0$ and any $x\in X$, we have $p(B_\varepsilon(x)) =B_\varepsilon(p(x))$. In other words, $p$ maps closed balls of radius $\varepsilon$ in $X$ onto closed balls of radius $\varepsilon$ in $Y$. Recall that a metric space $X$ is \emph{proper} if every closed ball in $X$ is compact. Every proper metric space is complete.

Let $(X,d)$ be a metric space and $\fol = \{L_\alpha\mid \alpha\in \Lambda\}$ a partition of $X$.

The foliation $\fol$ is \emph{equidistant} if for all leaves $L_\alpha,L_\beta\in \fol$ and all $x_\alpha \in L_\alpha$, we have $d(L_\alpha,L_\beta) = d(x_\alpha,L_\beta)$, where the first distance is the distance between subsets of $M$.
We will say that $\fol$ is a \emph{metric foliation} if it is an equidistant partition.
	
An element $L_\alpha\in \fol$ of a metric foliation is a \emph{leaf}. For any $x\in X$, we denote by $L_x$ the leaf of~$\fol$ containing $x$, and refer to $L_x$ as the \emph{leaf through $x$}. The \emph{leaf space} is the quotient space~${X/\fol}$ whose elements are the leaves of the foliation. As for singular Riemannian foliations, we set $X^* = X/\fol$. We define a metric on $X^*$ by letting $d(x^*,y^*) = d(L_x,L_y)$ for any $x^*,y^*\in X^*$. The leaf-projection map $p\colon X\to X^*$ is then a submetry.
	
When $X$ is a proper metric space, the leaves of a metric foliation of $X$ are closed subsets of $X$. Moreover, if $p\colon X\to Y$ is a submetry between metric spaces, then the partition $\fol = \big\{p^{-1}(y)\mid y\in Y\big\}$ is a metric foliation of $X$ with closed leaves (see \cite[p.~19]{Lytchak}).

\begin{lem}[\protect{\cite[Lemma~4.7]{Lytchak}}]
\label{L: characterization proper submetries}
Let $p\colon X\to Y$ be a submetry. If $X$ is a proper metric space, then the following assertions are equivalent:
\begin{enumerate}\itemsep=0pt
		\item[$(1)$] The map $p$ is proper.
		\item[$(2)$] The fibers of $p$ are compact.
		\item[$(3)$] There is a compact fiber.
\end{enumerate}
\end{lem}

Let $p\colon X	\to Y$ be a submetry between two metric spaces. Two points $x_1,x_2\in X$ are \emph{near $($with respect to $p)$} if $d(x_1,x_2) = d(p(x_1),p(x_2))$. The points $x_1$, $x_2$ are near if and only if they realize the distance between the leaves $L_{x_1}$ and $L_{x_2}$ in $X$.
	
Consider a length space $X$, a metric space $Y$, and a submetry $p\colon X\to Y$. Let $\gamma$ be a~geodesic in $Y$ through a point $y\in Y$. Given $x\in p^{-1}(y)$, there exists a geodesic $\tilde{\gamma}$ through $x$ that is mapped by $p$ isometrically onto $\gamma$. We call $\tilde{\gamma}$ the \emph{horizontal lift} of $\gamma$.	A geodesic in $X$ is \emph{horizontal } if it is mapped by $p$ isometrically onto a geodesic in $Y$. A shortest path between two points $x_1,x_2\in X$ is horizontal if and only if the points $x_1$, $x_2$ are near (see \cite[p.~21]{Lytchak}).

\begin{thm}[\protect{\cite[Theorem~7.2]{Lytchak}}]
\label{T: Fibers of submetry have the same metric properties as ambient space}
Let $p\colon X\to Y$ be a submetry between metric spaces. Then the following assertions hold:
\begin{enumerate}\itemsep=0pt
	\item[$(1)$] The components of each fiber $F_y = p^{-1}(y)$ are at distance at least $\varepsilon (y)$ from one another.
	\item[$(2)$] The intrinsic metric on each component of a fiber $F_y$ is locally Lipschitz-equivalent to the induced metric.
\end{enumerate}
\end{thm}

\subsection{Submetries between Alexandrov spaces}
\label{SS:submetries.between.alexandrov.spaces}
When the metric space $X$ is an Alexandrov space, we may gain further insight into the local structure of submetries. To do so, we first recall several notions and results that will enable us to describe the space of directions of a family of leaves. We follow \cite{Lytchak}.

Recall that a~map~${f\colon (X,d_X)\to (Y,d_Y)}$ between metric spaces is \emph{Lipschitz} if there exists a~real number~${K>0}$ such that $d_Y(f(x_1),f(x_2))\leq K d_X(x_1,x_2)$ for any $x_1,x_2\in X$. In this case, we say $f$ is \emph{$K$-Lipschitz}.

Let $\Sigma$ be an $n$-dimensional Alexandrov space with curvature bounded below by $1$.
Given $A\subset \Sigma$, we let
\[
\Pol(A) = \{v\in \Sigma \mid d_{\Sigma}(v,A)\geq \pi/2\}
\]
and refer to it as the \emph{polar set of $A$}.
Two points $v,w\in \Sigma$ are \emph{antipodal} if $d_{\Sigma}(v,w) =\pi$. Toponogov's comparison theorem implies that $\Pol(A)$ is a totally convex subset of $\Sigma$. Hence, $\Pol(A)$ is an Alexandrov space with $\curv\geq 1$.

Let $X$ and $Y$ be metric spaces, and let $CX$ and $CY$ be their respective topological cones. A~map $f\colon CX\to CY$ is \emph{homogeneous} if
$
	f([t,x]) = [t,f(x)]
$
for all $t\in [0,\infty)$, and all $x\in X$. As mentioned in Section~\ref{S: Alexandrov spaces}, if $\Sigma$ is an $n$-dimensional Alexandrov space with $\curv\geq 1$, then~$C\Sigma$ may be endowed with a metric $d$ such that $(C\Sigma,d)$ is an $(n+1)$-dimensional Alexandrov space of non-negative curvature. We call $(C\Sigma,d)$ the \emph{Euclidean cone} over $\Sigma$. We will refer to points in~$C\Sigma$ as \emph{vectors}and to points in $\Sigma$ as \emph{directions}.
Given a vector $v=[t,u]$ in $(C\Sigma,d)$, we will refer to $t$ as the \emph{magnitude} of $v$ 	and will set $|v| = t$. Let $p\colon C\Sigma \to C\Tau$ be a submetry between Euclidean cones. A vector $h\in C\Sigma$ is \emph{horizontal $($with respect to $p)$} if $|p(h)| = |h|$.
 A~subset~$A$ of an Alexandrov space $X$ is \emph{totally convex} if, for any two points $x,y\in A$, every (minimizing) geodesic joining $x$ and $y$ is contained in~$A$. If $\curv(X)\geq 1$, we require this only for pairs with~$d(x,y)<\pi$. In particular, $\Sp^0$ is always a totally convex subset in any round unit sphere.
	
\begin{prop}[\protect{\cite[Proposition 6.4 and Lemma 6.5]{Lytchak}}]
		\label{P: Horizontal and Vertical spaces of space of directions of submetry}
Let $\Sigma$ and $\Tau$ be Alexandrov spaces with $\curv\geq 1$ and let $C\Sigma$, and $C\Tau$ be their respective Euclidean cones. If $p\colon C\Sigma\to C\Tau$ is a~homogeneous submetry, then the following assertions hold:
\begin{enumerate}\itemsep=0pt
	\item[$(1)$] The preimage of the vertex $o\in C\Tau$ is a totally convex subcone $CV$ of $C\Sigma$ defined over a~totally convex subset $V\subset \Sigma$.
	\item[$(2)$] The cone $CH\subset C\Sigma$ over $H = \Pol(V)\subset \Sigma$, the polar set of $V$, is the set of horizontal vectors.
	\item[$(3)$] Let $H=\Pol(V)$ as in item $(2)$. Then $\Pol(H)=V$.
\end{enumerate}
\end{prop}

Let $V$ and $H$ be as in Proposition~\ref{P: Horizontal and Vertical spaces of space of directions of submetry}. 
Given $v\in V$, consider the set
\[
	H^v = \{h\in \Sigma \mid d_{\Sigma}(h,v) = \pi/2\}. 
\]
Since $v\in V\subset\Sigma$ and $d_\Sigma(h,v)=\pi/2$ implies $d_\Sigma(h,V)\ge\pi/2$,
we have $H^v\subset \Pol(V)=H$. Thus, every $h\in H^v$ is a horizontal direction.

The set $\widetilde{H}= \bigcap_{v\in V} H^v$ is the set of all points $h\in H$ that have distance equal to $\pi/2$ to any point in $V$. 
Since $H$ and $V$ are totally convex subsets of $\Sigma$, they are Alexandrov spaces of~${\curv \geq 1}$. Thus, their \emph{spherical join} $H*V$ is again an Alexandrov space with $\curv\geq 1$.

Moreover, there exists a $1$-Lipschitz map $P_{H}\colon \Sigma \to H\ast V$ for which the following assertions are equivalent.

\begin{prop}[\protect{\cite[Proposition 6.14 ]{Lytchak}}]
\label{P:LYTCHAK_SPANNING_SUBMETRY}
The following assertions are equivalent:
\begin{enumerate}\itemsep=0pt
\item[$(1)$] $P_H$ is surjective.
\item[$(2)$] $H = \widetilde{H}$.
\item[$(3)$] $P_H$ is a submetry.
\end{enumerate}
\end{prop}

If the $1$-Lipschitz map $P_H\colon \Sigma\to H\ast V$ satisfies any of the conditions in Proposition~\ref{P:LYTCHAK_SPANNING_SUBMETRY}, we say that the set $H\subset \Sigma$ is \emph{almost spanning}. Moreover, if $P_{H}$ is an isometry, we say that $H$ is \emph{spanning}. We call a homogeneous submetry $p\colon C\Sigma\to C\Tau$ \emph{$($almost$)$ spanning} if the horizontal set $H\subset \Sigma$ is (almost) spanning.
	
By considering the Euclidean cones over the spaces $\Sigma$, $H\ast V$, $H$, and $\Tau$, and the maps between them discussed above, we may represent a homogeneous submetry $p\colon C\Sigma\to C\Tau$ as a~composition
\[
	C\Sigma \overset{CP_H}{\longrightarrow} CH\times CV \overset{\mathrm{pr}_{CH}}{\longrightarrow} CH\overset{p}{\longrightarrow} C\Tau,
\]
where $\mathrm{pr}_{CH}$ is the projection onto the first factor of the metric product $CH\times CV$, and observe that $\mathrm{pr}_{CH}$ and $p$ are submetries.

\begin{lem}[\protect{\cite[Lemma~6.15]{Lytchak}}]
\label{L: submetry restricted to horizontal part is an isometry implies space of directions splits into horizontal and vertical}
Let $p\colon C\Sigma\to C\Tau$ be a homogeneous submetry and let $H\subset \Sigma$ be the set of horizontal directions.
If the restriction $\restr{p}{H}\colon H\to \Tau$ is an isometry, then $H$ is almost-spanning.
\end{lem}
	
A homogeneous submetry $p\colon C\Sigma\to C\Tau$ is \emph{regular} if $p\colon H\to \Tau$ is an isometry.

Following \cite[Definition 2.5]{Lytchak}, we will say that $\Sigma$ is \emph{Riemannian} if it is isometric to the unit round $n$-sphere~$\Sp^n$. If $\vol(\Sigma)>(1-\rho)\vol(\Sp^n)$ for a fixed and sufficiently small positive real number $\rho=\rho(n)$, we refer to $\Sigma$ as \emph{extremely thick}. If there exist $n+2$ points~${v_i\in \Sigma}$ with~${d(v_i,v_j)>\pi/2}$, we call $\Sigma$ \emph{thick}. Additionally, we say $\Sigma$ is \emph{round} if $\mathrm{rad}(\Sigma)>\pi/2$, where~$\rad(\Sigma)$, the \emph{radius} of $\Sigma$, is given~by
\[
 	\mathrm{rad}(\Sigma) = \inf_{v\in \Sigma}\sup_{w\in \Sigma}d(v,w).
\]
By Grove and Petersen's radius sphere theorem, every round Alexandrov space is homeomorphic to a sphere (see \cite{GrovePetersen1993}).

\begin{prop}[\protect{\cite[Proposition~6.16]{Lytchak}}]
\label{L: Base space of homogenous submetry is a disk, then p is regular}
Let $p\colon C\Sigma\to C\Tau$ be a homogeneous submetry. Then the following assertions hold:
\begin{enumerate}\itemsep=0pt
\item[$(1)$] If either $V$ or $H$ is a round space, then $p$ is spanning.
\item[$(2)$] If $\Tau$ is a round space, then $p$ is regular and spanning.
\end{enumerate}	
\end{prop}

We now recall the definition of differentiability for Lipschitz functions between Alexandrov spaces (cf.\ \cite[Section 3]{Lytchak}). Given a metric space $(X,d)$ and a real number $r>0$, we denote by~$rX$ the metric space $(X,rd)$.
Let $\{(X_i,x_i)\}_{i=1}^\infty$ and $(X,x)$ be pointed proper metric spaces.
The sequence $\{(X_i,x_i)\}$ \emph{converges in the pointed Gromov--Hausdorff sense to $(X,x)$} if, for each~${R>0}$, the sequence of closed balls $\{B_R(x_i)\}$ converges to $B_R(x)$ in the Gromov--Hausdorff sense. The following theorem gives an alternative characterization of the tangent cone of an Alexandrov space to the one in Section~\ref{S: Alexandrov spaces}.

\begin{thm}[\protect{\cite[Theorem~7.8.1]{BuragoGromovPerelman1992}}]
Let $X$ be an Alexandrov space, fix $x\in X$, and let $T_xX=C\Sigma_x(X)$ be the tangent cone to $X$ at $x$.
Then the pointed metric spaces $(\lambda X,x)$ converge in the pointed Gromov--Hausdorff sense to $(T_xX, 0)$ as $\lambda\to \infty$.
\end{thm}

Let $X$, and $Y$ be Alexandrov spaces, $U$ an open subset of $X$, and fix $x\in U$. Let $f\colon U\to Y$ be a Lipschitz map and let
$\{r_j\}$ be a sequence of positive real numbers such that $r_j\to 0$ as~${j\to\infty}$.
Then there exists a limit map
\[
f_{(r_j)} = \lim_\omega (f_j)\colon T_xX \to T_{f(x)}Y,
\]
where $f_j$ is the rescaled function
\[
f_j=f\colon\ \left(\frac{1}{r_j}X,x\right)\to \left(\frac{1}{r_j}Y,f(x)\right).
\]
We say that $f$ is \emph{differentiable at $x$}, if the map $f_{(r_j)}$ is independent of the sequence $\{r_j\}$. We call this uniquely defined Lipschitz function the \emph{differential of $f$ at $x$}, and denote it by ${\rm d}f_x\colon T_xX \to T_{f(x)}Y$. If the differential of $f$ exists at every point $x\in U$, then we say that $f$ is \emph{differentiable}. If $f$ is differentiable at $x\in U$, then its differential ${\rm d}f_x$ is homogeneous. That is, ${\rm d}f_x(tv) = t\,{\rm d}f(v)$ for each real number $t\geq 0$ and each $v\in T_xX$.

The following proposition shows that submetries may transfer geometric properties from the total space to the base space.

\begin{prop}[\protect{\cite[Proposition~4.4]{Lytchak}}]
\label{P: if X alex then Y alex}
Given a submetry $p\colon X\to Y$, the space $Y$ is an Alexandrov space when $X$ is an Alexandrov space.
\end{prop}
	
Moreover, a submetry $p\colon X\to Y$ from an Alexandrov space has a well-defined differential.

\begin{prop}[\protect{\cite[Proposition 5.1]{Lytchak}}]
Let $p\colon X\to Y$ be a submetry from an Alexandrov space $X$ to $Y$ a metric space. Then $p$ is differentiable and each differential ${\rm d}f_x\colon T_x X\to T_{p(x)} Y$ is a homogeneous submetry.
\end{prop}

Let $p\colon X\to Y$ be a submetry from an Alexandrov space $X$ to $Y$ a metric space. A vector $v\in T_xX$ is \emph{vertical} if ${\rm d}f_x(v)=0$. We call $v$ horizontal if $|{\rm d}f_x(v)|=|v|$. From the homogeneity of the differential, one may show that the set of vertical vectors forms a subcone $CV_x$ of $T_xX$, where $V_x\subset \Sigma_xX\subset T_xX$. Similarly, the set of horizontal vectors forms a subcone $CH_x$ of $T_xX$, where $H_x\subset \Sigma_xX\subset T_xX$.

A submetry $p\colon X\to Y$ between arbitrary Alexandrov spaces is \emph{regular} at a point $x\in X$ if the homogeneous submetry ${\rm d}p_x\colon T_xX \to T_{p(x)}Y$ is regular. We call the leaf $L_x$ a \emph{regular leaf}. Similarly, $p\colon X\to Y$ is \emph{$($almost$)$ spanning} at $x\in X$ if ${\rm d}p_x\colon T_xX \to T_{p(x)}Y$ is (almost) spanning. A point $x\in X$ is \emph{round} if $\Sigma_x(X)$ is a round space. The following corollary follows immediately from Proposition~\ref{L: Base space of homogenous submetry is a disk, then p is regular}.

\begin{cor}\label{L: submetry over a manifold point is regular}
A submetry $p\colon X \to Y$ from an Alexandrov space $X$ is regular over each round point $y\in Y$.
\end{cor}

The following theorem tells us that every Alexandrov space has a dense subset of points whose space of directions is a unit round sphere, and are therefore round points (in the sense defined before Proposition~\ref{P: Horizontal and Vertical spaces of space of directions of submetry}).

\begin{thm}[\protect{\cite[Theorem~10.8.5]{BuragoBuragoIvanov} and \cite[Corollary 10.9.13]{BuragoBuragoIvanov}}]
\label{T:geometrically.regular.points.are.dense}
If $X$ is an $n$-dimensional Alexandrov space, then
the set of points whose space of directions is a unit round sphere $($or, equivalently, whose tangent cone is Euclidean space$)$ is dense in $X$.
\end{thm}

Recall that, given a submetry $p\colon X\to Y$ between metric spaces, the intrinsic metric on the fiber $L = p^{-1}(y)$ is locally Lipschitz-equivalent to the induced metric on $L$ (see Theorem~\ref{T: Fibers of submetry have the same metric properties as ambient space}).

The following proposition allows us to identify, for any $x\in X$, the space of vertical directions~${V_x\subset \Sigma_xX}$ with $\Sigma_x(L)$, the space of di\-rections of the leaf $L$ through $x$.

\begin{prop}[\protect{\cite[Proposition~5.2]{Lytchak}}]
\label{P: Tangent space to leaf is cone of vertical space}
Let $p\colon X\to Y$ be a submetry between Alexandrov spaces. Then, for any fiber $L = p^{-1}(y)$ and any point $x\in L$, the tangent cone $T_x L $ is the cone~$CV_x$.
\end{prop}

Let $I$ be a closed interval. Given $f\colon X	\to Y$ and $\alpha\colon I\to Y$, a curve $\tilde{\alpha}\colon I\to X$ is a \emph{lift of $\alpha$} if $f\circ\tilde{\alpha}= \alpha$.

\begin{lem}[\protect{existence of geodesic lifts: \cite[Lemma~1]{GuijarroWalschap2011} and \cite[Lemma~2.1]{BeretovskiiGuijarro2000}}]
\label{L: Existence and uniqueness of lifts of geodesics}
Let $f\colon X \to Y$ be a submetry between Alexandrov spaces. For any geodesic $\alpha\colon [0, a] \to Y$, and any $x \in f^{-1}(\alpha(0))$, there is a geodesic lift $\tilde{\alpha}\colon [0, a] \to X$ of $\alpha$ starting at $x$, with length equal to that of $\alpha$. Moreover, if the geodesic $\alpha$ between $\alpha(0)$ and $\alpha(a)$ is unique, then the lift $\tilde{\alpha}$ is unique.
\end{lem}

\subsection{Holonomy maps of a submetry between Alexandrov spaces}\label{S: Holonomy of submetries}

Let $f\colon X\to Y$ be a submetry between Alexandrov spaces. Assume we have a geodesic $\alpha\colon [0,b]\to Y$ with $\alpha(0)=y$, $\alpha(b)=z$, and initial direction $w\in \Sigma_y Y$. By \cite[Lemma 5.4]{Lytchak} for $x\in f^{-1}(y)$ fixed, for each $h\in H_x$ with ${\rm d} f_x(h)=w$, there exists a unique geodesic $\tilde{\alpha}_{h}\colon [0,b]\to X$ starting at~$x$ with initial direction $h$, such that $f\circ\tilde{\alpha}_{h} = \alpha$.
If there exists a unique direction $\tilde{w}\in H_x$ such that $d f_x(\tilde{w})=w$, then we have a well-defined map $p_\alpha\colon f^{-1}(y)\to f^{-1}(z)$ given by
$
p_\alpha(x) = \tilde{\alpha}_{\tilde{w}}(b)$.
This map is called the \emph{holonomy map of $\alpha$}.

Consider a geodesic $\alpha\colon [0,b]\to Y$ such that, for each $x\in f^{-1}(\alpha(0))$ the initial direction of $\alpha$ has a unique lift in $H_x$ and $p_\alpha$ is well defined. As observed in \cite[Section 7.3]{Lytchak} for $\tilde{z}\in f^{-1}(\alpha(b))$, the set $p^{-1}_\alpha(\tilde{z})$ corresponds to the endpoints of all geodesics of $X$ starting at $\tilde{z}$ which are horizontal lifts of the geodesic $\alpha^{-1}(t)=\alpha(b-t)$ and whose initial direction is a horizontal lift of the initial direction of $\alpha^{-1}$.

\begin{prop}\label{L: holonomy maps are surjective and continuous}
Let $f\colon X\to Y$ be a submetry between Alexandrov spaces and
suppose $\alpha\colon [0,b]\to Y$ is a geodesic such that for each $x\in f^{-1}(\alpha(0))$ its initial direction has a unique lift in $H_x$, and the holonomy map $p_\alpha$ is well defined. Then the holonomy map $p_\alpha$ is continuous and surjective.
\end{prop}

\begin{proof}

We first prove the surjectivity of the holonomy map $p_\alpha$. Consider $z\in f^{-1}(\alpha(b))$. For the geodesic $\alpha^{-1}(t)= \alpha(b-t)$, there exists a horizontal lift \smash{$\widetilde{\alpha^{-1}}$} in $X$ with \smash{$\widetilde{\alpha^{-1}}(0)=z$}. Set \smash{$x=\widetilde{\alpha^{-1}}(b)$}, and observe that $f(x) = \alpha(0)$ by construction. The geodesic $\smash{\bigl(\widetilde{\alpha^{-1}}\bigr)^{-1}(t)} = \smash{\widetilde{\alpha^{-1}}(b-t)}$ is a horizontal lift of $\alpha$ starting at $x$. Since the horizontal lifts of $\alpha$ are unique by hypothesis, we conclude that \smash{$p_\alpha(x) = \widetilde{\alpha^{-1}}(0) = z$}.

We now prove that $p_\alpha$ is continuous. Fix $x\in f^{-1}(\alpha(0))$ and let $\{x_i\}_{i\in \N}\subset f^{-1}(\alpha(0))$ be a~sequence converging to $x$ in the intrinsic metric of the fiber. By Theorem~\ref{T: Fibers of submetry have the same metric properties as ambient space}, the intrinsic and the induced metric are locally Lipschitz equivalent on the fiber, hence also $d_X(x_i,x)\to 0$ as~${i\to \infty}$.
Let $\tilde{\alpha}_i\colon [0,b]\to X$ be the unique horizontal lift of $\alpha$ with $\tilde{\alpha}_i(0) = x_i$. Let $\tilde{\alpha}\colon [0,b]\to X$ be the unique lift of $\alpha$ with $\tilde{\alpha}(0) = x$. To simplify the notation, set $y_i = p_{\alpha}(x_i) = \tilde{\alpha}_i(b)$ and~${y = p_\alpha(x) = \tilde{\alpha}(b)}$. We will show that $y_i\to y$.

For each $t\in [0,b]$,
\[
d_X(\tilde{\alpha}_i(t),x)\leq d_X(\tilde{\alpha}_i(t),x_i)+d_X(x_i,x) \leq b+ d_X(x,x_i).
\]
Thus, for all sufficiently large $i$, we have that $\tilde{\alpha}_i([0,b])$ is contained in the closed ball $B_{b+1}(x)$, which is a compact subset of $X$. In particular, $\{y_i\}_{i\in\mathbb{N}}$ is contained in a compact subset of $X$.

Let $\{y_{i_j}\}_{i\in\mathbb{N}}$ be an arbitrary convergent subsequence with $y_{i_j}\to y'$ in $X$. Let us show that~${y' = y}$. Since the geodesics $\tilde{\alpha}_{i_j}\colon [0,b]\to B_{b+1}(x)\subset X$ have the same length and are contained in a compact subset of $X$, the Arzel\`a--Ascoli theorem (see \cite[Theorem 2.5.14]{BuragoBuragoIvanov}) implies that there exists a sub-subsequence (which we will not relabel) $\tilde{\alpha}_{i_j}$ that converges uniformly to a~curve $\beta\colon [0,b]\to X$, which must necessarily be a geodesic (see \cite[Proposition 2.5.17]{BuragoBuragoIvanov}). By the continuity of $f$,
\[
f\circ \beta = \lim_{j\to \infty}f\circ \tilde{\alpha}_{i_j} = \lim_{j\to\infty}\alpha = \alpha.
\]
Hence, $\beta$ is a horizontal lift of $\alpha$. Observe now that
\[
\beta(0) = \lim_{j\to\infty}\tilde{\alpha}_{i_j}(0) = \lim_{j\to\infty} x_{i_j} = x.
\]
Since we have assumed that lifts starting at the same point are unique, we must have $\beta = \tilde{\alpha}$. Taking endpoints, we get
\[
y' = \lim_{j\to\infty}y_{i_j} = \lim_{j\to \infty}\tilde{\alpha}_{i_j}(b) = \beta(b) = \tilde{\alpha}(b) = y.
\]
Therefore, every convergent subsequence of $\{y_i\}_{i\in \mathbb{N}}$ converges to $y$. Since $\{y_i\}_{i\in \mathbb{N}}$ is contained in a compact subset of $X$, it follows that $y_i\to y$, i.e., $p_{\alpha}(x_i) \to p_\alpha(x)$. Thus, $p_\alpha$ is continuous.
\end{proof}

\begin{cor}\label{L: holonomy maps over regular leaves are homeomorphisms}
Let $f\colon X\to Y$ be a submetry between Alexandrov spaces. If $\alpha\colon [0,b]\to Y$ is a geodesic whose endpoints are regular points, then the holonomy map $p_\alpha$ is well defined and is a homeomorphism.
\end{cor}

\begin{proof}
Let $y = \alpha(0)$ and $z=\alpha(b)$. By definition, for $x\in f^{-1}(y)$ we have $df_x\colon H_x\to \Sigma_y Y$ is an isometry. Thus, the holonomy map $p_\alpha$ is well defined. Moreover, for $\tilde{z}\in f^{-1}(z)$, the holonomy map $p_{\alpha^{-1}}$ is also well defined for the geodesic $\alpha^{-1}(t)= \alpha(b-t)$. Additionally, $p_{\alpha^{-1}}(p_{\alpha}(x))=x$ by construction. By Proposition~\ref{L: holonomy maps are surjective and continuous}, both $p_\alpha$ and $p_{\alpha^{-1}}$ are continuous. Thus, the conclusion follows.
\end{proof}

The following assertion is now an immediate consequence of Corollary~\ref{L: holonomy maps over regular leaves are homeomorphisms}.

 \begin{cor}\label{C: regular fibers are homeomorphic}
 Let $f\colon X\to Y$ be a submetry between Alexandrov spaces. Then the regular fibers of $f$ are homeomorphic.
 \end{cor}

\subsection{Isometry groups of Alexandrov spaces}
\label{S: Isometry group of an Alexandrov space}
As for Riemannian manifolds, the isometry group of an Alexandrov space admits a Lie group structure.
Additionally, there is an optimal upper bound on the dimension of the isometry group and a rigidity result holds when this bound is attained.

	\begin{thm}[\protect{\cite[Theorem 1.1]{FukayaYamaguchi1994} and \cite[Theorems 3.1 and 4.1]{GalazGarciaGuijarro2013}}]
	\label{T: Dimension of Isometry group}
Let $(X,d)$ be an $n$-dimensional Alexandrov space. Then the following assertions hold:
\begin{enumerate}[label=$(\arabic*)$]\itemsep=0pt
\item The group $G$ of isometries of $X$ is a Lie group and is compact if $X$ is compact.\label{T: Dimension of Isometry group i}
\item The dimension of $G$ satisfies $\dim (G) \leq n(n+1)/2$.
\item If $\dim (G) = n(n+1)/2$, then $X$ is a Riemannian manifold isometric to one of the following spaces of constant sectional curvature: the $n$-dimensional Euclidean space, an $n$-dimen\-sional round sphere, an $n$-dimensional round real projective space, or an $n$-dimensional hyperbolic space with constant negative sectional curvature.
		\end{enumerate}
\end{thm}

\begin{rem}\label{R: isometry group of local. Alexandrov space is Lie}
The conclusions of Theorem~\ref{T: Dimension of Isometry group}\,\ref{T: Dimension of Isometry group i} also hold for a complete length space $(X,d)$ which is \emph{locally an Alexandrov space}. That is, for each $x\in X$, there exist a neighborhood~${U_x\subset X}$ and a real number $k_x$ such that $(U_x,\restr{d}{U_x})$ is an Alexandrov space with $\mathrm{curv}\geq k$.
\end{rem}

Let $g\colon X\to X$ be an isometry of an Alexandrov space $X$, and suppose that there exists~${x\in X}$ such that $g(x) = x$. Then the differential ${\rm d}g_x\colon T_x X\to T_x X$ induces an isometry on the space of directions $\Sigma_x(X)\subset T_xX = C\Sigma_x(X)$. Let $G\subset \Iso(X)$ be a group acting on $X$ isometrically. Given $x\in X$, we define the \emph{orbit of $G$ through $x$} as $G(x) = \{g\cdot x\mid g\in G\}$. The \emph{isotropy subgroup at $x$} is the group $G_x = \{g\in G\mid g\cdot x = x\}$. The \emph{isotropy representation} of $G_x$ into $\Iso(\Sigma_x(X))$ is given by setting $g\cdot v = {\rm d} g_x(v)$ for $g\in G_x$ and $v\in \Sigma_x(X)$.
The action of $G$ on~$X$ is \emph{effective} if $\bigcap_{x\in X} G_x = \{e\}$. The following lemma implies that, if $G$ acts effectively on~$X$, then for any~${x\in X}$ the action of $G_x$ on $\Sigma_x(X)$ given by the isotropy representation is also effective.

\begin{lem}[\protect{\cite[Lemma~3.2]{GalazGarciaGuijarro2013}}]
\label{L: Isotropy action acts effectivelly}
Let $X$ be an Alexandrov space and let $g\colon X\to X$ an isometry. If there is some $x_0\in X$ such that $g(x_0)=x_0$ and ${\rm d}g_{x_0}\colon \Sigma_{x_0}(X)\to \Sigma_{x_0}(X)$ is the identity, then~${g(x)=x}$ for all $x\in X$.
\end{lem}

Let $Z$ be a connected locally compact metric space (in particular, any Alexandrov space).
The action of $\Iso(Z)$ on $Z$ is proper when $\Iso(Z)$ is equipped with the topology of pointwise convergence (see \cite[Proposition in Section 4, p.\ 11]{ManoussosStrantzalos2003}).
This topology agrees with the compact-open topology, which in turn is equivalent to the topology of uniform convergence over compact subsets (see, for example, \cite{Manoussos2010}).
Hence, the action is proper for the compact-open topology as well.
We record this fact for use in the proof of Theorem~\ref{T:MAIN_THM}, and note that one may also prove it for Alexandrov spaces by adapting the Riemannian proof (see, for example, \cite[Proposition~3.62]{AlexandrinoBettiol2015} and cf.\ \cite{DiazRamos2008a}).

\begin{prop}\label{P: Isometry group of Alex space acts properly}
Let $X$ be a connected Alexandrov space. Then the action of $\Iso(X)$ $($with the compact-open topology$)$ on $X$ is proper.
\end{prop}

\subsection{Foliated maps}
\label{s:foliated.maps}
To conclude this section, let us recall some results on foliated homeomorphisms (cf.\ \cite[Section~3]{MaysmenkoPolulyakh}). Let $p\colon X\to Y$ be a quotient map between two topological spaces $X$ and $Y$. Denote by~${\fol = \big\{p^{-1}(y)\mid y\in Y\big\}}$ the partition of $X$ induced by the pointwise preimages of $p$.
A continuous map $h\colon X\to X$ is \emph{foliated} if, for any $L\in \fol$, we have $h(L)\subset L'$ for some $L'\in \fol$.
Thus, every foliated map $h\colon X\to X$ induces a well-defined map $\Psi(h)\colon Y\to Y$
 given by $\Psi(h)(p(x)) = p(h(x))$
making the following diagram commute
\[
	\begin{tikzcd}
	X \arrow[d, "p", labels = left] \arrow[r,"h"]& X \arrow[d,"p"]\\
	Y \arrow[r,"\Psi(h)"] & Y.
	\end{tikzcd}
\]

We denote by $\Home(X,\fol)$ the group of all \emph{foliated homeomorphisms} (i.e., all foliated maps with an inverse map which is also foliated). We let $\Home(Y)$ be the group of all self-homeomorphisms of $Y$. We have a group homomorphism $\Psi\colon \Home(X,\fol)\to \Home(Y)$ given by $f\mapsto \Psi(f)$. The map $p\colon X\to Y$ \emph{admits local cross-sections} if, for any $y\in Y$, there exists an open neighborhood $V\subset Y$ and a continuous map $\sigma\colon V\to X$ such that $p\circ \sigma = \Id_{V}$.
When $Y$ is locally compact and Hausdorff, the following theorem guarantees the continuity of $\Psi$ when we equip both $\Home(X,\fol)$ and $\Home(Y)$ with the compact-open topology.
	
\begin{thm}[\protect{\cite[Corollary~3.4]{MaysmenkoPolulyakh}}]\label{T: Continuity of psi}
Let $X$ and $Y$ be topological spaces with $Y$ locally compact and Hausdorff. If $p\colon X \to Y$ is a quotient map, then each of the following conditions implies that the homomorphism $\Psi\colon \Home(X,\fol)\to \Home(Y)$ is continuous:
		\begin{enumerate}\itemsep=0pt
		 \item[$(1)$] The map $p$ is proper.
		 \item[$(2)$] The map $p$ is open and admits local cross-sections.
		 \item[$(3)$] The map $p$ is a locally trivial fibration.
		\end{enumerate}
	\end{thm}
	
Suppose now that $X$ is a metric space, the fibers of the quotient map $p\colon X\to Y$ are closed and equidistant (i.e., $d_{L_i}$ is constant on $L_j$, for $i, j = 1, 2$), and the quotient space $Y=X/\fol$ is equipped with the quotient metric
defined by
\[
d_Y(p(L_1),p(L_2))=\inf\{d_X(x_1,x_2)\mid x_i\in L_i\}
\]
for fibers $L_1,L_2\in\fol$. Denote by $\Iso(X,\fol)$ the group of \emph{foliated isometries}, i.e., isometries of~$X$ sending leaves to leaves.
Then the map $\Psi(f)\in \Home(Y)$ is an isometry of $Y$. Indeed, let~${y_i=p(L_i)}$ with leaves $L_i\in \fol$, and fix $x_i\in L_i$, $i=1,2$.
Then
\begin{align*}
d_Y\bigl(\Psi(f)(y_1),\Psi(f)(y_2)\bigr)
 &= d_Y\bigl(p(f(x_1)),p(f(x_2))\bigr)= d_X\bigl(f(L_1),f(L_2)\bigr) \\
 &= d_X(L_1,L_2) = d_Y\bigl(p(L_1),p(L_2)\bigr)
 = d_Y(y_1,y_2).
\end{align*}
Thus, $\Psi(f)\in \Iso(Y)$.
Consequently, $\Psi\colon\Iso(X,\fol)\to \Iso(Y)$ is a group homomorphism.

We conclude this section with the following general facts, which we will use in the proofs of Theorems~\ref{T:MAIN_THM_SRF} and \ref{T:MAIN_THM}.

\begin{prop}
\label{P: Iso(Y) closed implies Iso(X,F) closed}
Let $p\colon X\to Y$ be a continuous surjective map between metric spaces
and set~${\fol = \big\{p^{-1}(y)\mid y\in Y\big\}}$. Then the group $\Iso (X,\fol)$ of foliated isometries is closed in $\Iso(X)$.
\end{prop}

\begin{proof}

Recall that $\Iso(X)$ is equipped with the compact-open topology.
For each ordered pair $(x,x')\in X\times X$ with $p(x)=p(x')$, define
\[
\Phi_{x,x'}\colon \ \Iso(X) \to Y\times Y,
\qquad
g \mapsto \bigl(p(gx),\,p(gx')\bigr).
\]
In the compact-open topology, for each point $x\in X$, the \emph{Myers--Steenrod} map (i.e., the evaluation map)
\[
\mu_x\colon \ \Iso(X)\to X, \qquad g\mapsto gx
\]
is continuous (cf.~\cite{GalazGarciaGuijarro2013,GuijarroSantos2018}). Therefore, the map
$
\Phi_{x,x'} = (p\circ \mu_x) \times (p\circ \mu_{x'})
$
is continuous.

Since $Y$ is a metric space (hence Hausdorff), the diagonal
$
\Delta_Y = \{(y,y) \in Y\times Y\}
$
is closed in~${Y\times Y}$.
An isometry $g\in \Iso(X)$ is foliated if and only if, for all $x,x' \in X$,
$p(x)=p(x')$ implies that $p(gx)=p(gx')$,
i.e., if and only if $\Phi_{x,x'}(g)\in \Delta_Y$ for every pair $(x,x')$ with~${p(x)=p(x')}$. Therefore,
\[
\Iso(X,\fol) = \bigcap\Phi_{x,x'}^{-1}(\Delta_Y),
\]
where the intersection is taken over the set of all pairs $(x,x')\in X\times X$ with $p(x)=p(x')$. Since~$\Delta_Y$ is closed and $\Phi_{x,x'}$ is continuous, each preimage $\Phi_{x,x'}^{-1}(\Delta_Y)$ is closed in $\Iso(X)$. Hence, $\Iso(X,\fol)$ is a closed subgroup of $\Iso(X)$.
\end{proof}

Since a closed subgroup of a Lie group is itself a Lie group, Proposition~\ref{P: Iso(Y) closed implies Iso(X,F) closed} yields the following corollary.

\begin{cor}\label{C: Lie group Iso(Y) and closed implies Lie group Iso(M,F)}
Let $p\colon X\to Y$ be a continuous surjective map between metric spaces
and set~${\fol = \{p^{-1}(y)\mid y\in Y\}}$. If $\Iso(X)$ admits a Lie group structure with respect to the compact-open topology, then $\Iso (X,\fol)$ is a Lie group.
\end{cor}

\section[Proof of Theorem A]{Proof of Theorem~\ref{T:MAIN_THM_SRF}}\label{S: proof of Main corollary}
We will prove each item in Theorem~\ref{T:MAIN_THM_SRF} separately.

\subsection[Proof of Theorem A (i)]{Proof of Theorem~\ref{T:MAIN_THM_SRF}\,\ref{C:MAIN_THM_i}}
By the classical Myers--Steenrod theorem, $\Iso(M)$
is a Lie group and is compact if $M$ is compact. The conclusion now follows from Corollary~\ref{C: Lie group Iso(Y) and closed implies Lie group Iso(M,F)}.

\subsection[Proof of Theorem A (ii)]{Proof of Theorem~\ref{T:MAIN_THM_SRF}\,\ref{C:MAIN_THM_ii}}

Recall that $M_\prin\subset M$, the set of points in principal leaves, is an open and dense subset of~$M$ (see Section~\ref{S: Singular Riemannian foliations}).
Hence, $M_\prin$ is a Riemannian manifold foliated by the principal leaves of~$\fol$, and the quotient $M_\prin/\fol$ is also a Riemannian manifold.
Moreover, the leaf-projection map~${\pi\colon M_\prin\to M_\prin/\fol}$ is surjective and continuous.
By the classical Myers--Steenrod theorem, $\Iso(M_\prin)$ is a Lie group with respect to the compact-open topology. Hence, by Corollary~\ref{C: Lie group Iso(Y) and closed implies Lie group Iso(M,F)}, $\Iso(M_\prin,\fol)$ is a Lie group.

Consider the group homomorphism
\[
\Rho\colon \ \Iso(M,\fol)\to \Iso(M_{\prin},\fol)
\]
given by $\Rho(f) = f|_{M_{\prin}}$.
Note that $\Rho(f)\in \Iso(M_\prin,\fol)$, since foliated isometries send principal leaves to principal leaves.
Furthermore, $\Rho$ is continuous and, therefore, smooth by \cite[Corollary~3.50]{Hall2015}. Hence, $\Rho$ is a Lie group homomorphism.

Since $M_\prin/\fol$ is a Riemannian manifold, $\Iso(M_\prin/\fol)$ is a Lie group, by the classical Myers--Steenrod theorem.
Consider the group homomorphism
\[
\Psi_\prin\colon \ \Iso (M_\prin,\fol)\to\Iso (M_\prin/\fol),
\]
where $\Psi_\prin(f)$ is given by
$\Psi_\prin(f)(\pi(x)) = \pi(f(x))$ (cf.\ Section~\ref{s:foliated.maps}). The leaf-projection map~${\pi\colon M_{\prin}\to M_{\prin}/\fol}$ is a fiber bundle with $(n-k)$-dimensional fibers (see \cite[Theorem~A]{MendesRadeschi2019} and \cite[p.~41]{Corro}).
By Theorem~\ref{T: Continuity of psi}, $\Psi_\prin$ is continuous and, therefore, smooth. Hence $\Psi_\prin$ is a~Lie group homomorphism.

Consider now the Lie group homomorphism
\[
\Phi = \Psi_\prin\circ \Rho \colon \ \Iso(M,\fol) \to \Iso(M_\prin/\fol).
\]
By the first isomorphism theorem for Lie groups,
\begin{equation}
\label{eq:rank.nullity.Phi}
\dim(\Iso(M,\fol)) = \dim(\ker(\Phi)) + \dim(\Phi(\Iso(M,\fol)).
\end{equation}

Since $\Phi(\Iso(M,\fol))$ is a Lie subgroup of $\Iso(M_\prin/\fol)$ and $M_\prin/\fol$ is a $k$-dimensional Riemannian manifold, by \cite[Chapter~II, Theorem~3.1]{Kobayashi},
\begin{equation}
\label{eq:bound.dim.im.phi}
\dim(\Phi(\Iso(M,\fol)))\leq \dim(\Iso(M_\prin/\fol)) \leq \frac{k(k+1)}{2}.
\end{equation}

Let $G = \ker(\Phi)\leq \Iso(M,\fol)$ and observe that $g\in G$ if and only if $g(L) = L$ for any principal leaf $L$.
Fix $p\in M_\prin$ and write $T_pM = V_p\oplus H_p$, where $V_p=T_pL$ is the vertical space and $H_p = V_p^\perp$ is the horizontal space.
For any $g\in G$, we have $\pi\circ g = \pi$, which implies that~${{\rm d}\pi_p\circ {\rm d}g_p = {\rm d}\pi_p}$.
Since $\pi\colon M_\prin\to M_\prin/\fol$ is a Riemannian submersion, ${\rm d}\pi_p|_{H_p}$ is a linear isomorphism onto $T_{\pi(p)}(M_\prin/\fol$), and therefore ${\rm d}g_p|_{H_p}=\mathrm{Id}$. Hence, $G_p$, the isotropy group at~$p$, acts trivially on $H_p$ and orthogonally on $V_p$.
Therefore,
\begin{equation}
\label{eq:bound.dim.G_p}
\dim(G_p)\leq \dim(\mathrm{O}(V_p)) \leq \frac{(n-k)(n-k-1)}{2}.
\end{equation}
Since the orbit $G(p)$ of $G$ through $p$ is contained in the principal leaf containing $p$,
\begin{equation}
\label{eq:bound.dim.G(p)}
\dim(G(p))\leq n-k.
\end{equation}
Since $\dim(G) = \dim(G(p)) + \dim(G_p)$, inequalities \eqref{eq:bound.dim.G_p} and \eqref{eq:bound.dim.G(p)} imply that
 \begin{align}
 \label{eq:bound.dim.ker.phi}
 \dim(\ker (\Phi)) = \dim(G)
 \leq n-k +\frac{(n-k)(n-k-1)}{2} = \frac{(n-k)(n-k+1)}{2}.
 \end{align}

Combining equation~\eqref{eq:rank.nullity.Phi} with inequalities~\eqref{eq:bound.dim.im.phi} and \eqref{eq:bound.dim.ker.phi}, we obtain
\begin{equation}
\label{eq:bound.dim.Iso(M,F)}
\dim(\Iso(M,\fol)) \leq \frac{k(k+1)}{2} + \frac{(n-k)(n-k+1)}{2},
\end{equation}
thus verifying inequality \eqref{EQ:DIM_BOUND}.

\subsection[Proof of Theorem A (iii)]{Proof of Theorem~\ref{T:MAIN_THM_SRF}\,\ref{C:MAIN_THM_iii}}

Assume equality holds in inequality \eqref{EQ:DIM_BOUND}, i.e.,
\[
\dim(\Iso(M,\fol)) = \frac{k(k+1)}{2} + \frac{(n-k)(n-k+1)}{2}.
\]
Then, by equation~\eqref{eq:rank.nullity.Phi} and inequalities~\eqref{eq:bound.dim.im.phi} and \eqref{eq:bound.dim.ker.phi},
we must have
\begin{equation}
\label{eq:equality.dim.im.Phi}
\dim(\Phi(\Iso(M,\fol))) = \frac{k(k+1)}{2},
\end{equation}
and
\begin{equation}
\label{eq:equality.dim.ker.Phi}
\dim(\ker (\Phi))=\frac{(n-k)(n-k+1)}{2}.
\end{equation}

Inequality~\eqref{eq:bound.dim.im.phi} and equation~\eqref{eq:equality.dim.im.Phi} imply that
$
\dim(\Iso(M_\prin/\fol)) = k(k+1)/2$.
Since $M_\prin/\fol$ is a $k$-dimensional Riemannian manifold, by \cite[Chapter~II, Theorem 3.1]{Kobayashi}, $M_\prin/\fol$ is isometric to one of the following $k$-dimensional space forms: Euclidean space, hyperbolic space, a round sphere, or round real projective space. In particular, $M_\prin/\fol$ is complete. Since $M/\fol$ is the metric completion of $M_\prin/\fol$, it follows that $M/\fol = M_\prin/\fol$. Therefore, the foliation is regular and all the leaves are principal.

Since $\ker(\Phi)$ acts effectively by isometries on any leaf and leaves are $(n-k)$-dimensional, equation~\eqref{eq:equality.dim.ker.Phi} implies that the isometry group of any leaf has maximal possible dimension. Hence, by \cite[Chapter~II, Theorem~3.1]{Kobayashi}, that all the leaves are isometric to a round $\Sp^{n-k}$, a~round~$\RP^{n-k}$, Euclidean space $\R^{n-k}$, or a hyperbolic space $\mathbb{H}^{n-k}$.

Recall that the leaf-projection map $\pi\colon M_{\prin}\to M_{\prin}/\fol$ is a smooth Riemannian submersion.
Since all the leaves are principal, $\pi\colon M\to M/\fol$ is a fiber bundle, and thus the desired conclusion follows.

\section[Proof of Theorem B]{Proof of Theorem~\ref{T:MAIN_THM}}\label{S:PROOF_MAIN_THM}
	
We will prove each item in Theorem~\ref{T:MAIN_THM} separately. Throughout this section, we let $\pi\colon X\to Y$ be a~submetry between Alexandrov spaces and set $\fol = \big\{\p^{-1}(y)\mid y\in Y\big\}$. In the proofs of items~\ref{T:MAIN_THM_ii} and \ref{T:MAIN_THM_iii} we assume $\pi$ is proper.

\subsection[Proof of Theorem B (i)]{Proof of Theorem~\ref{T:MAIN_THM}\,\ref{T:MAIN_THM_i}}
By Theorem~\ref{T: Dimension of Isometry group}\,\ref{T: Dimension of Isometry group i}, $\Iso(X)$ is a Lie group and is compact if $X$ is compact. The conclusion now follows from Corollary~\ref{C: Lie group Iso(Y) and closed implies Lie group Iso(M,F)}.

\subsection[Proof of Theorem B (ii)]{Proof of Theorem~\ref{T:MAIN_THM}\,\ref{T:MAIN_THM_ii}}
We now compute the upper bound on the dimension of the Lie group $\Iso(X,\fol)$. Recall that from now on and in the rest of the section we assume that the submetry $\pi\colon X\to Y$ is proper.
Suppose that $X$ has dimension $n\geq 1$ and $Y$ has dimension $0\leq k \leq n$, so that the regular leaves of the foliation $\fol$ have dimension $n-k$. We will show that
\[
 	\dim(\Iso(X,\fol)) \leq \frac{k(k+1)}{2}+\frac{(n-k)(n-k+1)}{2}.
\]

Recall that a point $y\in Y$ is \emph{regular with respect to $\pi$} if $\pi$ is regular at each $x\in \pi^{-1}(y)$ (see~\cite[Definition 6.5]{Lytchak}).
The set of regular points
\[
Y_{\mathrm{reg}}=\{y\in Y\mid y \mbox{ is regular with respect to }\pi\}
\]
contains the open and dense subset of thick points (see \cite[end of Section~6]{Lytchak}).
Consequently, the set
$
X_{\mathrm{reg}}= \{x\in X\mid x \mbox{ is regular}\}
$
contains an open and dense subset.
Let
\[
R_X = \big\{x\in X\mid \Sigma_x X \mbox{ is isometric to the unit round } \Sp^{n-1}\big\}.
\]
By Theorem \ref{T:geometrically.regular.points.are.dense}, $R_X$ is dense. Hence, $R_X\cap X_{\mathrm{reg}}\neq \varnothing$.

Fix $\bar{x}\in R_X\cap X_{\mathrm{reg}}$.
We may choose $\bar{x}$ so that \smash{$\Sigma_{\pi(\bar{x})} Y$} is isometric to the unit round $\Sp^{k-1}$, since the set of points in $Y$ whose space of directions is isometric to the unit round $\Sp^{k-1}$ is dense in $Y$ and $Y_{\mathrm{reg}}$ contains an open subset.
Since $\bar{x}$ is regular,
\[
{\rm d}_{\bar{x}} \pi|_{H_{\bar{x}}}\colon\ H_{\bar{x}}\to \Sigma_{\pi(\bar{x})} Y
\]
is an isometry (see Section~\ref{SS:submetries.between.alexandrov.spaces}). Moreover, since $\Sigma_{\pi(\bar{x})} Y$ is round, there exists an isometry between $\Sigma_{\bar{x}} X$ and the spherical join $H_{\bar{x}}* V_{\bar{x}}$. Since $\Sigma_{\bar{x}} X$ is isometric to the unit round $\Sp^{n-1}$ and~$H_{\bar{x}}$ is isometric to the unit round $\Sp^{k-1}$, we conclude that $V_{\bar{x}}$ is isometric to the unit round~$S^{n-k-1}$.

Since $Y$ is an Alexandrov space, $\Iso(Y)$ is a Lie group, by Theorem~\ref{T: Dimension of Isometry group}. Consider the group homomorphism
$
\Psi\colon \Iso (X,\fol) \to \Iso (Y)$,
where $\Psi(f)$ is given by
$\Psi(f)(\pi(x)) = \pi(f(x))$ (cf.\ Section~\ref{s:foliated.maps}). We assume $\pi\colon X\to Y$ is proper. Hence, Theorem~\ref{T: Continuity of psi} implies that $\Psi$ is continuous, and thus it is a Lie group homomorphism.

Set $G = \ker(\Psi)$. Given $f\in G$, we have $f(\bar{x})\in \pi^{-1}(\pi(\bar{x}))$, allowing us to identify the differential ${\rm d}_{\bar{x}} f \colon \Sigma_{\bar{x}} X\to \Sigma_{f(\bar{x})} X$ (which is an isometry) with an isometry
\[
({\rm d}_{\bar{x}}f)_1 * ({\rm d}_{\bar{x}}f)_2\colon\ H_{\bar{x}}* V_{\bar{x}}\to H_{f(\bar{x})}* V_{f(\bar{x})}.
\]
Since $G$ acts on $\pi^{-1}(\pi(\bar{x}))$, we may consider the isotropy group $G_{\bar{x}}$.
For $f\in G_{\bar{x}}$, we have then that $({\rm d}_{\bar{x}}f)_1=\mathrm{Id}_{H_{\bar{x}}}$ and $({\rm d}_{\bar{x}}f)_2\colon V_{\bar{x}}\to V_{\bar{x}}$ is an isometry.
Thus, $G_{\bar{x}}$ acts by isometries on the unit round $S^{n-k-1}$.
Moreover, this action is effective.
Indeed, if $({\rm d}_{\bar{x}}f)_2 =\mathrm{Id}_{V_{\bar{x}}}$, then ${\rm d}_{\bar{x}}f = \mathrm{Id}_{\Sigma_{\bar{x}} X}$.
Since $f(\bar{x}) = \bar{x}$ and $f$ is an isometry of the Alexandrov space $X$, then \cite[Lemma~3.2]{GalazGarciaGuijarro2013} implies that $f=\mathrm{Id}_X$.
Hence, by the Myers--Steenrod theorem,
\[
 \dim (G_{\bar{x}})\leq \frac{(n-k-1)(n-k)}{2}.
\]
Now we consider the orbit $G(\bar{x})\subset \pi^{-1}(\pi(\bar{x}))$.
By \cite[Korollar 7.5]{Lytchak}, the Hausdorff dimension of~$\pi^{-1}(\pi(\bar{x}))$ is $n-k$.
Thus, we conclude that the dimension of $G(\bar{x})= G/G_{\bar{x}}$ is bounded above by $n-k$. It follows that
\begin{align}
\dim(G)= \dim(G(\bar{x}))+\dim(G_{\bar{x}}) \leq n-k+\frac{(n\hspace{-0.5pt}-\hspace{-0.5pt}k\hspace{-0.5pt}-\hspace{-0.5pt}1)(n\hspace{-0.5pt}-\hspace{-0.5pt}k)}{2}
= \frac{(n\hspace{-0.5pt}-\hspace{-0.5pt}k)(n\hspace{-0.5pt}-\hspace{-0.5pt}k\hspace{-0.5pt}+\hspace{-0.5pt}1)}{2}.\!\!\!\!\label{eq:bound.alex.1}
\end{align}

Next, observe that $\Psi(\Iso(X,\fol))$ is a Lie subgroup of $\Iso(Y)$. Since $Y$ is an Alexandrov space of dimension $k$, Theorem~\ref{T: Dimension of Isometry group} implies that
\begin{align}\label{eq:bound.alex.2}
\dim(\Psi(\Iso(X,\fol))) \leq k(k+1)/2.
\end{align}
Then, by the first isomorphism theorem for Lie groups, inequalities~\eqref{eq:bound.alex.1} and \eqref{eq:bound.alex.2}, and recalling that $G=\ker(\Psi)$,
\begin{align}
\dim ( \Iso (X,\fol) )  = \dim (G) + \dim (\Psi(\Iso(X,\fol)) ) \leq \frac{(n\hspace{-0.5pt}-\hspace{-0.5pt}k)(n\hspace{-0.5pt}-\hspace{-0.5pt}k\hspace{-0.5pt}+\hspace{-0.5pt}1)}{2} +\frac{k(k\hspace{-0.5pt}+\hspace{-0.5pt}1)}{2}.\!\!\!\label{eq:ineq.bound.dim.iso.Alex.proof}
\end{align}

\begin{rem}
The conclusions of Theorem~\ref{T:MAIN_THM_SRF}\,\ref{C:MAIN_THM_ii} do not follow directly from the conclusions of Theorem~\ref{T:MAIN_THM}\,\ref{T:MAIN_THM_ii}, since the leaf space of a closed singular Riemannian foliation $\fol$ on a complete manifold $M$ is only locally an Alexandrov space of bounded curvature (see \cite[p.~119]{LytchakThorbergsson2010}). This is because the local lower curvature bound on $M/\fol$ is given by a local lower sectional curvature bound on $M$. When $M$ is not compact, it may happen that $M$ does not have a global lower sectional curvature bound, and thus in this case $M/\fol$ does not have a lower curvature bound in the sense of Alexandrov.
\end{rem}

\begin{rem}
Due to the lack of a slice theorem for submetries between Alexandrov spaces, we need properness of $\pi\colon X\to Y$ to guarantee that $\Psi\colon \Iso(X,\fol)\to \Iso(Y)$ is continuous, and thus a Lie group homeomorphism, to be able to say that $\dim(\Iso(X,\fol))=\dim(\ker\Psi)+\dim(\img(\Psi))$.
\end{rem}

\subsection[Proof of Theorem B]{Proof of Theorem~\ref{T:MAIN_THM}\,\ref{T:MAIN_THM_iii}}Recall our standing assumption that $\pi\colon X\to Y$ is a proper submetry between Alexandrov spaces.
Assume that equality holds in~\eqref{eq:ineq.bound.dim.iso.Alex.proof}, i.e.,
\[
\dim(\Iso(X,\fol)) = \frac{k(k+1)}{2}+\frac{(n-k)(n-k+1)}{2},
\]	
where $\dim(X) = n$ and $\dim(Y) = k$ with $0\leq k\leq n$. We first determine the topology of the base space $Y$.

\begin{lem}\label{L: equality implies base is a manifold isometric to model spaces}
The base space $Y$ is isometric to $k$-dimensional Euclidean space, hyperbolic space with constant negative sectional curvature, a round real projective space, or a round sphere.
\end{lem}

\begin{proof}
Since $Y$ is $k$-dimensional, Theorem~\ref{T: Dimension of Isometry group} implies that $\dim(\Iso(Y))\leq k(k+1)/2$.
Since $\dim(\Iso(X,\fol))$ is maximal, it follows from the proof of Theorem~\ref{T:MAIN_THM}\,\ref{T:MAIN_THM_ii} that the subgroup~${\Psi(\Iso(X,\fol))\subset \Iso(Y)}$ has dimension $k(k+1)/2$. Thus, $\Iso(Y)$ has dimension $k(k+1)/2$ and the conclusion follows from Theorem~\ref{T: Dimension of Isometry group}.
\end{proof}

We now determine the topology of the fibers. Consider the Lie group homomorphism
$
\Psi\colon  \Iso(X,\fol)\to \Iso(Y)
$
with $\Psi(f)$ given by $\Psi(f)(\pi(x)) = \pi(f(x))$ and set $G=\ker(\Psi)\subset \Iso(X,\fol)$.

\begin{lem}\label{L: fibers contain homogeneous subsets}
For any $x\in X$, the orbit $G(x)\subset \pi^{-1}(\pi(x))$ is a compact subset.
\end{lem}

\begin{proof}Fix $x\in X$ and let $d_I$ be the intrinsic metric on $L_{x}=\pi^{-1}(\pi(x))$ induced by $d|_{L_{x}\times L_{x}}$.
We first prove that $G\subset\Iso(L_{x},d_I)$.
Let $\gamma\colon [0,1]\to L_x$ be a rectifiable curve.
Since $G\subset \Iso(X)$, for any $g\in G$,
$L(g\circ \gamma)=L(\gamma)<\infty$,
where $L(\cdot)$ denotes length.
By definition,
\[
d_I(x,y)=\inf\{L(\gamma)\mid \gamma\colon[0,1]\to L_x \text{ is rectifiable with } \gamma(0)=x,\, \gamma(1)=y\}.
\]
Thus, $d_I(g(x),g(y))\leq d_I(x,y)$. Since $g$ is arbitrary,
\[
d_I(x,y)=d_I\bigl(g^{-1}(g(x)),g^{-1}(g(y))\bigr)\leq d_I(g(x),g(y)).
\]
Thus, $G\subset \Iso(L_{x},d_I)$.

By construction, $G(x)\subset L_{x}$.
We now show that $G(x)$ is closed.
Consider $\{g_i\}_{i\in \N}\subset G$ with~${\lim_{i\to\infty}g_i(x)=y\in X}$.
By Proposition~\ref{P: Isometry group of Alex space acts properly}, the action of $\Iso(X)$ on $X$ is proper. Then, there exists a subsequence $\{g_{i_k}\}_{k\in \N}$ converging to some $g\in \Iso(X)$ (see \cite[Proposition~21.5]{Lee}, whose proof is purely topological, and cf.\ \cite[Proposition~2.2]{CorroKordass2019} or \cite{DiazRamos2008a}).
Recall that the map $\Psi$ is continuous. Hence, $G=\ker(\Psi)$ is closed in $\Iso(X,\fol)$, and $\Iso(X,\fol)$ is closed in $\Iso(X)$.
Therefore, $g\in G$.
By the continuity of the action of $G$ on $X$, $y=g(x)\in G(x)$.
Thus, the $G$-orbits are closed in $X$. Since $\pi\colon X\to Y$ is proper, $L_x$ is compact, and since $G(x)\subset L_x$ is closed, $G(x)$ is also compact.
\end{proof}

\subsection*{Connected fiber case} We first consider the case where $\pi$ has connected fibers.

\begin{lem}\label{L: fibers are homogeneous Riemannian manifolds}
If all fibers of $\pi\colon X\to Y$ are connected, then $($with the intrinsic metric$)$ each fiber is isometric to a round sphere or a round real projective space.
\end{lem}

\begin{proof}

As stated in the proof of Theorem~\ref{T:MAIN_THM}\,\ref{T:MAIN_THM_ii}, we may fix $\bar{x}\in X$ such that $\Sigma_{\bar{x}}(X)$ is isometric to a unit round $S^{n-1}$ and $\Sigma_{\pi(\bar{x})}(Y)$ is isometric to a unit round $S^{k-1}$.
From our hypotheses and the proof of Theorem~\ref{T:MAIN_THM}\,\ref{T:MAIN_THM_ii}, we observe that
$
\dim_H(G(\bar{x})) = n-k = \dim_H(L_{\bar{x}})$,
where $L_{\bar{x}} = \pi^{-1}(\pi(\bar{x}))$ is the fiber containing $\bar{x}$ and $\dim_H(\cdot)$ denotes Hausdorff dimension.

Endow the orbit space $X/G$ with the orbital distance $d^\ast$ and consider the submetry
$
\pi_G\colon (X,\allowbreak d_X)\to (X/G,d^\ast)$.
Since the orbits of $G$ are contained in the fibers of $\pi$ \big(i.e., $G(x)\subset \pi^{-1}(\pi(x))$ for any $x\in X$\big), we obtain a well-defined map
$\pi_D\colon (X/G,d^\ast)\to (Y,d_Y)$
with $\pi = \pi_D\circ \pi_G$. Since~$\pi$ and $\pi_G$ are submetries, $\pi_D$ is also a submetry.
Hence,
$\dim_H(X/G) = n-(n-k) = k$.
This implies that the submetry $\pi_D$ has discrete fibers, and thus $\pi_D^{-1}(\pi(x))= \bigsqcup_{i\in \N} \{x^\ast_i\}$. Since~${\pi = \pi_D\circ \pi_G}$, it follows that
$
L_{\bar{x}} = \bigsqcup_{i\in \N} G(x_{i})$.
By \cite[Theorem 7.2]{Lytchak}, the connected components of each fiber of $\pi$ have positive distance to one another.
By hypothesis, $L_{\bar{x}}$ is connected. Hence, for $G^0\subset G$, the connected component of the identity, we have that $G(\bar{x})=G^0(\bar{x})$, $\pi=\pi_G$, and
$
L_{\bar{x}}=G^0(\bar{x})$.
Thus, $L_{\bar{x}}$ with the intrinsic metric is a compact homogeneous inner metric space.

{\samepage
Since $\Iso(X)$ is a Lie group, $G^0$ is a Lie subgroup. Hence, $L_{\bar{x}}=G^0(\bar{x}) = G^0/\bigl(G^0\bigr)_{\bar{x}}$ is a~manifold and, in particular, locally compact and locally contractible. We also have
$\dim_{\mathrm{top}}L_{\bar{x}} = \dim_HL_{\bar{x}} = n-k$.
By \cite{BerestovskiiVershik1992} (cf.\ \cite[Theorem 13]{Berestovskii2014}), $(L_{\bar{x}},d_I)$ is a homogeneous Finsler manifold. Moreover, at $\bar{x}$, the space of directions $\Sigma_{\bar{x}}L_{\bar{x}}$ is isometric to the unit round $\Sp^{n-k-1}$ and the metric tangent cone is isometric to the $(n-k)$-dimensional Euclidean space. Since $G^0$ acts transitively and by isometries on $L_{\bar{x}}$, this holds at every point in $L_{\bar{x}}$. Therefore, the Finsler norm is Euclidean, and it follows that $(L_{\bar{x}},d_I)$ is a homogeneous Riemannian manifold (see also~\cite[Theorem~7]{Berestovskii1989}).

}

Therefore, $L_{\bar{x}}$ is a union of Riemannian manifolds and, hence, a compact Riemannian manifold with finitely many connected components.
By hypothesis, the isometry group of $L_{\bar{x}}$ has maximal dimension $(n-k)(n-k+1)/2$.
Hence, by \cite[Chapter~II, Theorem 3.1]{Kobayashi}, $L_{\bar{x}}$ is isometric to a~round~$\Sp^{n-k}$ or a round $\RP^{n-k}$.

By Lemma~\ref{L: equality implies base is a manifold isometric to model spaces}, the space $Y$ is a Riemannian manifold.
Thus, every point $y\in Y$ is a regular point of $\pi=\pi_G$.
From the definition of regular point, we conclude that given $x^\ast\in X/G$, the set of horizontal directions $H_{x^\ast}(\pi_D)$ of $\pi_D$ is isometric to $\Sigma_{\pi(x^\ast)}Y = \Sp^{k-1}$.
Moreover, $H_{x^\ast}(\pi_D) = \Sigma_{x^\ast}(X/G)$. This implies that all fibers of $\pi=\pi_G$ are regular.
By Corollary~\ref{C: regular fibers are homeomorphic}, each fiber of $\pi$ is homeomorphic to $L_{\bar{x}}=G(\bar{x})$.
That is, all the fibers are homogeneous spaces isometric to round spheres or round real projective spaces.
\end{proof}

\begin{lem}\label{L: Isom(X,F) acts transitively}
If the fibers of $\pi\colon X \to Y$ are connected, then $\Iso(X,\fol)$ acts transitively on $X$.
\end{lem}

\begin{proof}
By Theorem~\ref{T:MAIN_THM}\,\ref{T:MAIN_THM_ii} and equality in
\eqref{eq:ineq.bound.dim.iso.Alex.proof}, the image $\Psi(\Iso(X,\fol))\subset \Iso(Y)$ has the same dimension as $\Iso(Y)$.
Hence, $\Psi(\Iso(X,\fol))$ contains $\Iso(Y)^0$, the identity component of $\Iso(Y)$.
By Lemma~\ref{L: equality implies base is a manifold isometric to model spaces}, $Y$ is isometric to $\mathbb{R}^k$, $\mathbb{H}^k$, a round $S^k$ or a round $\RP^k$.
Thus, $\Iso(Y)^0$ is isomorphic to $\R^k\rtimes \mathrm{SO}(k)$ (Euclidean space), $\mathrm{SO}(k,1)$ (hyperbolic space), or ${\mathrm{SO}(k+1)}$ (sphere and real projective space).
In each case, $\Iso(Y)^0$ acts transitively on $Y$.
Therefore, $\Psi(\Iso(X,\fol))$ acts transitively on $Y$.

Let $x,\bar{x}\in X$ and choose $h\in \Iso(X,\fol)$ with
$\Psi(h)(\pi(x))=\pi(\bar{x})$.
Then, by the definition of $\Psi$, we have $h(x)\in\pi^{-1}(\pi(\bar{x}))$.
Since $G\subset \Iso(X,\fol)$ acts transitively on $\pi^{-1}(\pi(\bar{x}))$ by the proof of Lemma~\ref{L: fibers are homogeneous Riemannian manifolds}, there exists $g\in G=\ker(\Psi)$ with $g(h(x))=\bar{x}$.
Thus, we conclude that~$\Iso(X,\fol)$ acts transitively on $X$.
\end{proof}

\begin{lem}
\label{L:connected.case.submetry.is.a.riemannian.submersion}
If all fibers of $\pi\colon X\to Y$ are connected, then $X$ and $Y$ are Riemannian manifolds, and $\pi\colon X\to Y$ is a smooth Riemannian submersion.
\end{lem}

\begin{proof}
By Lemma~\ref{L: Isom(X,F) acts transitively}, $X$ is a homogeneous Alexandrov space. Hence, by \cite[Theorem 7]{Berestovskii1989}, $X$ is isometric to a homogeneous Riemannian manifold $(M,g)$.
By Lemma~\ref{L: fibers are homogeneous Riemannian manifolds}, each fiber is an orbit $G^0(x)$ of $G=\ker(\Psi)\subset \Iso(X,\fol)$.
Since $X$ is now Riemannian and $G\subset \Iso(X)$ acts by isometries, each fiber is a smooth embedded submanifold of $X$.
Finally, by \cite[Lemma 13.1]{LytchakWilking2024} implies that $\pi\colon X\to Y$ is a smooth map.
\end{proof}

\subsection{Disconnected fiber case} Suppose now that the fibers of $\pi\colon X\to Y$ have one or more connected components.
Recall our standing assumption that the submetry $\pi\colon X\to Y$ is proper and equality in~\eqref{eq:ineq.bound.dim.iso.Alex.proof} holds. Recall that by Lemma~\ref{L: equality implies base is a manifold isometric to model spaces}, $Y$ is a $k$-dimensional Riemannian manifold isometric to Euclidean space, hyperbolic space, a round sphere, or a round real projective space. We distinguish two cases, depending on whether $Y$ is compact.

\begin{lem}
\label{L:disconnected.fibers.Y.is.Rk}
If the fibers of $\pi\colon X\to Y$ are possibly disconnected and $Y$ is isometric to Euclidean or hyperbolic space, then $X$ is homeomorphic to $Y\times F$ with $F$ either a round sphere or a round real projective space.
\end{lem}

\begin{proof}

By the factorization of submetries between Alexandrov spaces \cite[Theorem 10.1]{Lytchak}, there exist submetries
$\tilde{\pi}\colon X\to Z$ with connected fibers and $\pi_D\colon Z\to Y$ with discrete fibers such that~${\pi = \pi_D\circ\tilde{\pi}}$.

Set
$\tilde{\fol} =\big\{\tilde{\pi}^{-1}(\tilde{\pi}(x))\mid x\in X\big\}$.
Let us verify that $\Iso(X,\fol)\subset \Iso(X,\tilde{\fol})$.
Fix $y\in Y$ and let $f\in \Iso(X,\fol)$.
Since $f$ preserves $\pi^{-1}(y)$ and
\[
\pi^{-1}(y) = \bigsqcup_{z\in \pi^{-1}_D(y)}\tilde{\pi}^{-1}(z)
\]
with each $\tilde{\pi}^{-1}(z)$ a connected component of $\pi^{-1}(y)$, $f$ must map $\tilde{\pi}^{-1}(z)$ to some $\tilde{\pi}^{-1}(z')$.
Hence, $f\in\Iso(X,\tilde{\fol})$.

Since
$\dim(Z) = \dim(Y) = k$,
we have
\begin{align*}
\frac{k(k+1)}{2}+\frac{(n-k)(n-k+1)}{2}=\dim(\Iso(X,\fol))&\leq \dim\bigl(\Iso(X,\tilde{\fol})\bigr)\\
&\leq \frac{k(k+1)}{2}+\frac{(n-k)(n-k+1)}{2},
\end{align*}
where the upper bound follows from Theorem~\ref{T:MAIN_THM}\,\ref{T:MAIN_THM_ii}.
Thus, $\Iso(X,\tilde{\fol})$ has maximal dimension.

By Lemma~\ref{L: fibers are homogeneous Riemannian manifolds}, each fiber of $\tilde{\fol}$ (with the intrinsic metric) is isometric to a round sphere or a round real projective space. Therefore, each fiber of $\pi$ is a disjoint union of round spheres or real projective spaces.
Since $\pi$ is proper, each fiber $\pi^{-1}(y)$ is compact, hence consists of finitely many connected components.

By Lemma~\ref{L:connected.case.submetry.is.a.riemannian.submersion}, $\tilde{\pi}\colon X\to Z$ is a smooth Riemannian submersion.
Since $\pi_D\colon Z\to Y$ has discrete fibers, $\pi_D$ is a (Riemannian) covering map (see \cite[Theorem 1.2]{Lange2020}). Since $Y$ is simply-connected, the covering $\pi_D\colon Z\to Y$ is trivial. Thus, $Z = Y$ and $\pi_D$ is an isometry. Hence, $\pi = \tilde{\pi}$ is a smooth Riemannian submersion with connected fibers (see also \cite{LytchakWilking2024}). In particular, $\pi\colon X\to Y$ is a fiber bundle with fiber $F$ a single round sphere or real projective space.
Since $Y$ is contractible, the bundle is trivial and $X$ must be homeomorphic to $Y\times F$.
\end{proof}

\begin{lem}
If the fibers of $\pi\colon X\to Y$ are possibly disconnected and $Y$ is $S^k$ or $\RP^k$, then for~${k\geq 2}$ either each fiber has exactly two connected components and $Y$ is isometric to a~round~$\RP^{k}$, or each fiber is connected and $Y$ is isometric to a round $\Sp^k$ or a round $\RP^k$. For~${k=1}$, the fibers can have $m\geq 1$ connected components.
\end{lem}

\begin{proof}
Factor $\pi=\pi_D\circ\tilde{\pi}$ with $\tilde{\pi}\colon X\to Z$ a submetry with connected fibers, and $\pi_D\colon Z\to Y$ a~submetry with discrete fibers.
Arguing as in the proof of Lemma~\ref{L:disconnected.fibers.Y.is.Rk}, each fiber of $\tilde{\pi}$ is isometric to a round sphere or a round real projective space and $\pi_D\colon Z\to Y$ is a Riemannian covering map with $Y$ isometric to a round $S^k$ or a round $\RP^k$.
As $Y$ is connected, the number of sheets of the covering~$\pi_D$ is constant. Therefore, the number of connected components of the fibers of~$\pi$ is constant on $Y$.

If $k\geq 2$, then either $Z=Y = S^k$ and $\pi_D$ is an isometry, or $Z=S^k$ and $Y=\RP^k$ with $\pi_D$ the standard two-fold covering.
If $\pi_D$ is an isometry, then $\pi = \tilde{\pi}$ and the fibers are connected.
Hence, $\pi = \tilde{\pi}$ is a smooth Riemannian submersion. In particular, $\pi\colon X\to Y$ is a fiber bundle, with fiber a round sphere or real projective space and base a sphere. If $\pi_D$ is the two-fold cover $S^k\to \RP^k$, then for each $y\in Y$, $\pi^{-1}_D$ consists of two points and $\pi^{-1}(y)$ has exactly two connected components.

If $k=1$, in addition to the identity $S^1\to S^1$ and the two-fold covering $S^1\to \RP^1$, one must also consider the standard $m$-fold Riemannian coverings $S^1\to S^1$ given by cyclic rotation groups of order $m$.
\end{proof}

\subsection*{Acknowledgements}
We thank Alexander Lytchak and Marco Radeschi for helpful comments on a preliminary version of this article. We thank the organizers of the IV joint meeting of RSME and SMM, and the Universidad Polit\'ecnica de Valencia, for their hospitality while this manuscript was finished. We thank the anonymous referees for suggestions that improved both clarity and accuracy. In particular, we are grateful to one of the referees for observations that helped strengthen Theorems~\ref{T:MAIN_THM_SRF} and \ref{T:MAIN_THM} and simplify their proofs, for pointing out the argument in the proof of Proposition~\ref{P: Iso(Y) closed implies Iso(X,F) closed}, and for proposing the question highlighted in Remark~\ref{R:classification.problem}.
D.~Corro was supported in part by UNAM-DGAPA Postdoctoral fellowship of the Institute of Mathematics, and by the DFG (grant CO 2359/1-1, Priority Programme SPP2026 ``Geometry at Infinity''), and by a UKRI Future Leaders Fellowship [grant number MR/W01176X/1; PI: J Harvey].
F.~Galaz-Garc\'ia was supported in part by the DFG (grant GA 2050 2-1, Priority Programme SPP2026 ``Geometry at Infinity'').

\pdfbookmark[1]{References}{ref}
\LastPageEnding

\end{document}